\documentclass[article]{IEEEtran}
\usepackage[top=1in, bottom=1in, left=0.75in, right=0.75in]{geometry}

\usepackage{amsmath, amssymb, epsfig, graphicx, bm}
\usepackage{cases}
\usepackage{float, subfigure}
\usepackage{color}
\usepackage{amsthm}
\usepackage{amsfonts}
\usepackage{cite, url}
\usepackage[all,cmtip]{xy}
\usepackage{color,hyperref}
\usepackage{algorithmic,algorithm}

\definecolor{darkblue}{rgb}{0,0,1}
\hypersetup{colorlinks,breaklinks,
linkcolor=darkblue,urlcolor=darkblue,anchorcolor=darkblue,citecolor=darkblue}

\usepackage[normalem]{ulem} 

\newtheorem{theorem}{Theorem}

\newtheorem{lemma}{Lemma}
\newtheorem{remark}{Remark}

\newtheorem{assumption}{Assumption}
\newtheorem{corollary}{Corollary}

\allowdisplaybreaks[4]%
\begin{document}

\title{Communication-Censored Linearized ADMM for Decentralized Consensus Optimization}
\author{\authorblockN{Weiyu Li, Yaohua Liu, Zhi Tian, and Qing Ling}
\thanks{Weiyu Li is with the School of Gifted Young, University of Science and Technology of China, Hefei, Anhui 230026, China. Yaohua Liu is with the Department of Automation, University of Science and Technology of China, Hefei, Anhui 230026, China. Zhi Tian is with the Department of Electrical and Computer Engineering, George Mason University, Fairfax, VA 22030, USA. Qing Ling is with the School of Data and Computer Science, Sun Yat-Sen University, Guangzhou, Guangdong 510006, China. Zhi Tian is supported by NSF grants CCF-1527396 and IIS-1741338. Qing Ling is supported by China NSF grants 61573331 and 11728105. Part of this paper has appeared in the 44th IEEE International Conference on Acoustics, Speech, and Signal Processing, Brighton, UK, May 12--17, 2019. Corresponding Email: lingqing556@mail.sysu.edu.cn.}}

\maketitle

\begin{abstract}
	
In this paper, we propose a communication- and computation-efficient algorithm to solve a convex consensus optimization problem defined over a decentralized network. A remarkable existing algorithm to solve this problem is the alternating direction method of multipliers (ADMM), in which at every iteration every node updates its local variable through combining neighboring variables and solving an optimization subproblem. The proposed algorithm, called as \underline{co}mmunication-censored \underline{l}inearized \underline{A}DMM (COLA), leverages a linearization technique to reduce the iteration-wise computation cost of ADMM and uses a communication-censoring strategy to alleviate the communication cost. To be specific, COLA introduces successive linearization approximations to the local cost functions such that the resultant computation is first-order and light-weight. Since the linearization technique slows down the convergence speed, COLA further adopts the communication-censoring strategy to avoid transmissions of less informative messages. A node is allowed to transmit only if the distance between the current local variable and its previously transmitted one is larger than a censoring threshold. COLA is proven to be convergent when the local cost functions have Lipschitz continuous gradients and the censoring threshold is summable. When the local cost functions are further strongly convex, we establish the linear (sublinear) convergence rate of COLA, given that the censoring threshold linearly (sublinearly) decays to $0$. Numerical experiments corroborate with the theoretical findings and demonstrate the satisfactory communication-computation tradeoff of COLA.


\end{abstract}



\section{Introduction}
\label{sec:intro}

In this paper, we consider solving a convex consensus optimization problem
\begin{align}\label{eq:obj}
\tilde{x}^* = \arg\min\limits_{\tilde{x}} \sum_{i=1}^n
f_i(\tilde{x}),
\end{align}
which is defined over a bidirectionally connected decentralized network consisting of $n$ nodes. All the nodes cooperate to find an optimal argument $\tilde{x}^*$ of the common optimization variable $\tilde{x} \in \mathcal{R}^p$, but the convex local cost function $f_i(\tilde{x}):\mathcal{R}^p \rightarrow \mathcal{R}$ held by every node $i$ is kept private. We focus on the scenario that the nodes are unable to afford complicated computation, while the communication resources are also limited. Our goal is to devise a communication-efficient decentralized algorithm, which relies on light-weight computation, to solve \eqref{eq:obj}.

Decentralized consensus optimization has attracted extensive interest in recent years. Problems in the form of \eqref{eq:obj} are involved in a variety of research areas,
including wireless sensor networks \cite{Rabbat2004-ipsn, Schizas2008-1,sensor1}, communication networks \cite{zeng2011, Giannakis2015-ADMM}, multi-robot networks \cite{Bullo2009, Cao2013-TII}, smart grids \cite{Giannakis2013,grid1,Liu2017}, machine learning systems \cite{cluster1,Mokhtari-Thesis,Lian2017}, to name a few. Popular algorithms to solve \eqref{eq:obj} span from the primal domain to the dual domain. The primal domain algorithms, such as sub-gradient descent \cite{Nedic2009, Jakovetic2014, Yuan2016}, dual averaging \cite{Duchi2012, Tsianos2012,Lee2017} and network Newton \cite{Mokhtari2017-NN}, have to use diminishing step sizes to guarantee exact convergence to an optimal solution, and thus suffer from slow convergence. On the other hand, \eqref{eq:obj} can be reformulated as a constrained optimization problem and solved by dual domain algorithms, among which the celebrated alternating direction method of multipliers (ADMM) is able to achieve fast and exact convergence \cite{Schizas2008-1, Boyd2010, Mateos2010, Shi2014-ADMM}. When ADMM is implemented in a synchronous manner, at every iteration, every node solves an optimization subproblem dependent on its local cost function, and then exchanges the calculated local variable with its neighbors. Therefore, if the local cost functions are not in simple forms, solving the subproblems is computationally demanding. To alleviate the computation cost, the decentralized linearized ADMM (DLM) replaces the local cost functions in ADMM by their linear approximations, and attains a dual domain method with light-weight computation \cite{Ling2015-DLM, Chang2015}. Similar techniques have also been applied to develop other first-order dual domain algorithms, such as EXTRA \cite{Shi2015}, NEXT \cite{Lorenzo2016}, and gradient tracking methods \cite{Sun2016, Qu2017, Nedic2017, Xin2018, Pu2018}. If computing the inverse of a Hessian matrix is affordable at a node, one can replace the local cost functions by their quadratic approximations. The resultant second-order algorithms, DQM and ESOM, have faster convergence than their first-order counterparts \cite{Mokhtari2016-DQM, Mokhtari2016-ESOM}. Between the first- and second-order algorithms, a recent work in \cite{Eisen2018} develops a primal-dual quasi-Newton method that approximates the second-order information with local gradients. The lower complexity bounds and rate-optimal algorithms of decentralized optimization are developed in \cite{Scaman2017, Scaman2018, Sun2018}. Note that the communication cost in the aforementioned algorithms is proportional to the number of iterations, since after a given number of iterations every node needs to communicate with its neighbors.

In all decentralized algorithms, there is an essential communication-computation tradeoff \cite{Pereira2012, Berahas2017, Lan2017, Nedic20176, Tang2018}. An algorithm with light-weight iteration-wise computation generally needs more number of iterations, and in consequence more communication cost, to reach a target accuracy. For example, compared with ADMM, DLM enjoys simple gradient-based computation, but suffers from relatively slow convergence speed and high communication cost. In this paper, we aim at achieving a favorable communication-computation tradeoff in a decentralized network, where the nodes are only affordable to light-weight gradient-based computation. The limitation on the computation power may come from that the nodes are equipped with cheap computing units in a wireless sensor network, or from that using higher-order information is prohibitively time-consuming for finding a high-dimensional solution in a machine learning system.

Given the constraint on the computation cost, we adopt the communication-censoring strategy to further save the communication cost. The basic idea of the communication-censoring strategy is to only allow transmissions of informative messages over the network. A simple yet powerful protocol is to prevent a node from transmitting a variable that is close to its previously transmitted one, where the ``closeness'' is determined by comparing the Euclidean distance with a predefined time-varying censoring threshold. The communication-censoring strategy is tightly related to event-triggered control of continuous-time networks \cite{Johansson2012,Garcia2013,Cameron2016}, and finds successful applications in discrete-time decentralized optimization \cite{Lu2017, Tsianos2013, Chen2016, COCA}. It has been combined with primal domain methods such as sub-gradient descent \cite{Lu2017} and dual averaging \cite{Tsianos2013}, as well as dual domain methods such as dual decomposition \cite{Chen2016} and ADMM \cite{COCA}. However, similar to their uncensored counterparts, the primal domain methods in \cite{Lu2017, Tsianos2013} have to use diminishing step sizes to guarantee exact convergence. On the other hand, the dual domain methods in \cite{Chen2016, COCA} require the nodes to solve computationally demanding subproblems. Our proposed algorithm, called as communication-censored linearized ADMM (COLA), combines the communication-censoring strategy with the first-order dual domain method DLM. Particularly, we modify the standard communication-censoring strategy in \cite{Lu2017,Tsianos2013,Chen2016,COCA} to fit for the special algorithmic structure of DLM so as to attain better performance. We rigorously establish convergence as well as sublinear and linear convergence rates of COLA. To the best of our knowledge, COLA is the first communication-censored method that only uses gradient information but achieves linear convergence.

Starting from the derivation of the classical ADMM in Section \ref{sec:algo-admm}, we introduce COLA in Section \ref{sec:algo-cola}. COLA modifies ADMM in two aspects. First, linearizing the local cost functions enables approximately solving the time-consuming subproblems in ADMM, and thus saves computation. Second, the communication-censoring strategy is applied to remedy the poor communication efficiency caused by the linearization step. To further demonstrate the design principles of COLA, its tradeoff between communication and computation is discussed and compared with those of several existing dual domain algorithms in Section \ref{sec:algo-tradeoff}. In Section \ref{sec:conv}, we prove that when the censoring threshold is properly chosen, COLA converges to an optimal solution of \eqref{eq:obj} (Theorem \ref{theorem:convergence}). Moreover, when the local cost functions are strongly convex, the linear and sublinear convergence rates of COLA are established (Theorems \ref{theorem:linear} and \ref{theorem:sublin}). The analysis provides guidelines for choosing the parameters of COLA to reduce computation and communication costs. Section \ref{sec:nume} presents numerical experiments and demonstrates the communication-computation tradeoff of COLA. Section \ref{sec:con} summarizes our work.


\textit{Notation.} For matrices $A\in\mathcal{R}^{a\times n}$ and $B\in\mathcal{R}^{b\times n}$, $[A;B]\in\mathcal{R}^{(a+b)\times n}$ stacks the two matrices by rows. Define the inner product of two vectors $v_1$ and $v_2$ as $\langle v_1,v_2 \rangle:=v_1^Tv_2$, which naturally induces the Euclidean norm $\|v\|:=\sqrt{\langle v,v \rangle}$ of a vector $v$. For a matrix $M$, define $\lambda_{\min}(M)$ as the smallest eigenvalue, $\sigma_{\max}(M)$ as the largest singular value, and $\tilde{\sigma}_{\min}(M)$ as the smallest nonzero singular value. When $M$ is a block matrix, $(M)_{i,j}$ denotes its $(i,j)$-th block.

Throughout the paper, we consider a bidirectionally connected network $\mathcal{G}=\{\mathcal{V},\mathcal{A}\}$, where $\mathcal{V}=\{1,\ldots,n\}$ denotes the set of $n$ nodes
and $\mathcal{A}=\{1,\ldots,m\}$ is the set of $m$ directed arcs. Nodes $i$ and $j$ are called as neighbors if $(i,j)\in\mathcal{A}$ and $(j,i)\in\mathcal{A}$. We denote the set of node $i$'s neighbors as $\mathcal{N}_i$ with cardinality $d_{ii} =|\mathcal{N}_i|$. Further define the extended block arc source matrix $A_s\in\mathcal{R}^{mp\times np}$ containing $m\times n$ square blocks $(A_s)_{e,i}\in\mathcal{R}^{p\times p}$. The block $(A_s)_{e,i}=I_p$ if the arc $e=(i,j)\in\mathcal{A}$ and is null otherwise, where $I_p$ is the $p$-dimensional identity matrix. Likewise, define the extended block arc destination matrix $A_d\in\mathcal{R}^{mp\times np}$, whose block $(A_d)_{e,j}\in\mathcal{R}^{p\times p}$ is not null but $I_p$ if and only if the arc $e=(i,j)\in\mathcal{A}$ terminates at node $j$. Then, define the extended oriented incidence matrix as $G_o=A_s-A_d$ and the unoriented one as $G_u=A_s+A_d$. The oriented Laplacian is written as $L_o=\frac{1}{2}G_o^T G_o$ and the unoriented Laplacian $L_u=\frac{1}{2}G_u^T G_u$. The degree matrix is defined as $D=\frac{1}{2}(L_o+L_u)$, which is block diagonal with diagonal blocks $D_{i,i} = d_{ii}I_p$.

\section{Algorithm Development}
\label{sec:algo}

In this section, we propose COLA, the communication-censored linearized ADMM to solve the decentralized consensus optimization problem \eqref{eq:obj}. Rooted on ADMM, COLA features in two ingredients, \textit{linearization} to reduce the computation cost and \textit{communication censoring} to reduce the communication cost. We shall first introduce the development of ADMM in Section \ref{sec:algo-admm}, and then combine the linearization and communication-censoring techniques to devise COLA in Section \ref{sec:algo-cola}. The tradeoff between computation and communication is discussed in Section \ref{sec:algo-tradeoff}.

\subsection{ADMM: Alternating Direction Method of Multipliers}
\label{sec:algo-admm}

ADMM is a powerful tool to solve a structured optimization problem with two blocks of variables, which are separable in the cost function and subject to a linear equality constraint. To rewrite \eqref{eq:obj} into the standard bivariate form, we introduce local variables $x_i\in\mathcal{R}^p$ as copies of $\tilde{x}$ at nodes $i$, and auxiliary variables $z_{ij}\in\mathcal{R}^p$ at arcs $(i,j)\in\mathcal{A}$. Since the network is connected, \eqref{eq:obj} is equivalent to
\begin{align}\label{eq:obj-ADM-node}
\min\limits_{\{x_i\}, \{z_{ij}\}} \quad &\sum_{i=1}^{n}f_i(x_i), \nonumber\\
\rm{s.t.} \quad &x_i=z_{ij},\ x_j=z_{ij}, ~ \forall (i,j)\in\mathcal{A}.
\end{align}
An optimal solution of \eqref{eq:obj-ADM-node} satisfies $x_i^*=\tilde{x}^*$ and $z_{ij}^*=\tilde{x}^*$, where $\tilde{x}^*$ is an optimal solution of \eqref{eq:obj}.

Concatenate the variables as $x=[x_1;\ldots;x_n]\in\mathcal{R}^{np}$ and $z=[z_1;\ldots;z_m]\in\mathcal{R}^{mp}$, introduce the aggregate function $f(x):=\sum_{i=1}^{n}f_i(x_i)$, and denote $A:=\left[ A_s; A_d \right]\in\mathcal{R}^{2mp\times np}$ and $B:=\left[ -I_{mp} ; -I_{mp} \right]$. The matrix form of \eqref {eq:obj-ADM-node} is
\begin{equation}\label{eq:obj-ADM}
\min \limits_{x, z}\ f(x), \quad {\rm s.t.} \ Ax+Bz=0,
\end{equation}
which is the standard bivariate form handled by ADMM, except that the variable $z$ is absent in the cost function.

Introduce the augmented Lagrangian of \eqref{eq:obj-ADM} as $$L(x,z,\lambda)=f(x)+ \langle \lambda, Ax+Bz \rangle+\frac{c}{2}\|Ax+Bz\|^2,$$ where the penalty parameter $c>0$ is an arbitrary positive constant and the Lagrange multiplier $\lambda:=[\phi;\psi]\in\mathcal{R}^{2mp}$. The two vectors $\phi$, $\psi\in\mathcal{R}^{mp}$ are the Lagrangian multipliers
associated with the two constraints $A_sx-z=0$  and $A_dx-z=0$ respectively. At time $k$, the ADMM update follows
\begin{align}
x^{k+1}       & = \arg\min_x L(x,z^k,\lambda^k), \nonumber \\
z^{k+1}       & = \arg\min_z L(x^{k+1},z,\lambda^k), \nonumber \\
\lambda^{k+1} & = \lambda^{k} + c(Ax^{k+1}+Bz^{k+1}). \nonumber
\end{align}

According to \cite{Shi2014-ADMM}, if the variables are initialized with $\phi^0=-\psi^0$ and $G_ux^0=2z^0$, then we can eliminate $z^{k+1}$ and replace $\lambda^{k+1}$ by a lower-dimensional dual variable, such that the update is reduced to
\begin{align}\label{admm-x}
x^{k+1}      & = \arg\min_x\ f(x) + \langle \mu^k - cL_u x^k, x\rangle + cx^TDx, \\
\mu^{k+1}    & = \mu^k + cL_o x^{k+1}, \label{admm-mu}
\end{align}
where $\mu^k := G_o^T \phi^k \in \mathcal{R}^{np}$. By splitting $\mu^k = [\mu_1^k, \ldots, \mu_n^k]$, $\mu_i^k \in \mathcal{R}^p$ denotes the local dual variable of node $i$.

Using the definitions of $f(x)$, $D$, $L_u$ and $L_o$, we describe how the decentralized ADMM is implemented. At time $k$, every node $i$ updates its local primal variable $x_i^{k+1}$ using its $x_i^k$ and $\mu_i^k$, as well as $x_j^k$ from all neighbors $j$ via
\begin{align}\label{admm-x-node}
x_{i}^{k+1}  =& \arg\min\limits_{x_i} f_i(x_i) + \langle \mu_i^k - c \sum_{j \in \mathcal{N}_i} (x_i^k + x_j^k), x_i \rangle+ cd_{ii} x_i^2.
\end{align}
Then node $i$ broadcasts its $x_i^{k+1}$ to all neighbors. Finally, node $i$ updates its local dual variable $\mu_i^{k+1}$ using its $x_i^{k+1}$ and $\mu_i^k$, as well as $x_j^{k+1}$ from all neighbors $j$ via
\begin{equation}
\mu_i^{k+1}  = \mu_i^{k} + c\sum_{j\in \mathcal{N}_i} (x_{i}^{k+1}-x_{j}^{k+1}).\label{admm-mu-node}
\end{equation}

The costs of implementing ADMM are two-fold. The first is in computing the local primal and dual variables $x_i^{k}$ and $\mu_i^{k}$, in which the update of $x_i^{k}$ in \eqref{admm-x-node} is particularly demanding when the local cost function $f_i(x_i)$ is complicated. The second is in transmitting the local primal variables $x_i^{k+1}$, which is expensive when the bandwidth resource is limited.

\subsection{COLA: Communication-Censored Linearized ADMM}
\label{sec:algo-cola}

COLA adopts two strategies to improve the computation and communication efficiency of ADMM: linearization and communication censoring. The linearization technique has been used in \cite{Ling2015-DLM, Chang2015} to devise DLM, a gradient-based variant of ADMM. DLM effectively reduces the computation cost of solving subproblems in ADMM, but sacrifices on the convergence speed and thus results in high communication cost. Therefore, we use the communication-censoring strategy to prevent transmissions of less informative messages. Note that though the communication-censoring strategy has been applied to improve the communication efficiency of sub-gradient descent, dual averaging, dual decomposition and ADMM \cite{Lu2017,Tsianos2013,Chen2016,COCA}, we customize it in COLA so as to achieve a satisfactory balance between communication and computation, as we shall explain below.

%

\textit{Linearization.} Notice that the update of the primal variable $x_i^{k+1}$ in \eqref{admm-x-node}, which usually has no explicit solution, dominates the computation cost of ADMM. Therefore, a computationally demanding inner loop should be used to solve $x_i^{k+1}$. To address this issue, \cite{Ling2015-DLM, Chang2015} linearizes the local cost functions at every iteration. To be specific, at time $k$, the function $f_i(x_i)$ in \eqref{admm-x-node} is replaced by its quadratic approximation $f_i(x_i^k) + \langle \nabla f_i(x_i^k), x_i-x_i^k \rangle + \frac{\rho}{2}\|x_i-x_i^k\|^2$ at $x_i=x_i^k$, where $\rho > 0$ is a positive linearization parameter. Therefore, the primal variable is updated via
\begin{align}\label{dlm-x-node}
x_{i}^{k+1}  =& x_i^k - \frac{1}{2cd_{ii}+\rho}\big(  \nabla f_i(x_i^k) + c\sum_{j\in \mathcal{N}_i} (x_{i}^{k}-x_{j}^{k}) + \mu_i^{k} \big).
\end{align}
Note that the main computation cost of \eqref{dlm-x-node} is in calculating the gradient $\nabla f_i(x_i^k)$, which is light-weight. The update of dual variable remains the same as \eqref{admm-mu-node} in ADMM.

\textit{Communication censoring.} The linearization technique significantly reduces the computation cost of ADMM, but slows down the convergence speed, and hence results in high communication cost. Hence, we introduce the communication-censoring strategy to further reduce the communication cost. Intuitively, when $x_i^{k+1}$ is close to $x_i^k$, it is not necessary for node $i$ to transmit both of them to neighbors. Motivated by this fact, the communication-censoring strategy prevents transmissions of less informative messages so as to reduce the communication cost.

To rigorously explain the communication-censoring strategy, define a state variable $\hat{x}_i^{k} \in \mathcal{R}^p$ as the latest value that node $i$ has transmitted to neighbors before time $k$. At time $k$, after calculating $x_i^{k+1}$, node $i$ evaluates the difference between $\hat{x}_i^k$ and $x_i^{k+1}$ by their Euclidean distance
%
$\xi_i^{k+1} = \| \hat{x}_i^{k} - x_i^{k+1} \|, $
%
and then compares the difference with a predefined censoring threshold $\tau^{k+1} \geq 0$. Node $i$ is allowed to transmit $x_i^{k+1}$ to neighbors and update $\hat{x}_i^{k+1}=x_i^{k+1}$, if and only if
%
$\xi_i^{k+1} \geq \tau^{k+1}.$
%
Otherwise, the transmission is censored and $\hat{x}_i^{k+1}=\hat{x}_i^{k}$. With the state variable $\hat{x}_i^{k}$, COLA changes the DLM updates in \eqref{dlm-x-node} and \eqref{admm-mu-node} to
\begin{align}\label{cola-x-node}
x_{i}^{k+1}  &= x_{i}^{k} - \frac{1}{2cd_{ii}+\rho}  \big( \nabla f_i(x_{i}^{k})
+ c\sum\limits_{j \in \mathcal{N}_i} (\hat{x}_{i}^{k} - \hat{x}_{j}^{k}) + \mu_i^k \big),  \\
\mu_i^{k+1} &= \mu_i^{k} + c\sum_{j\in \mathcal{N}_i} (\hat{x}_{i}^{k+1}-\hat{x}_{j}^{k+1}). \label{cola-mu-node}
\end{align}
Stacking the state variables in $\hat{x}=[\hat{x}_1;\ldots;\hat{x}_n]\in\mathcal{R}^{np}$, we can write \eqref{cola-x-node} and \eqref{cola-mu-node} in the matrix form of
\begin{align}\label{cola-x}
\hspace{-1em}x^{k+1}&=x^k - (2cD+\rho I)^{-1} \left(\nabla f(x^k)  + cL_o\hat{x}^{k} + \mu^k \right), \\
\hspace{-1em}\mu^{k+1} &= \mu^k + cL_o \hat{x}^{k+1}.\label{cola-mu}
\end{align}

COLA run by node $i$ is outlined in Algorithm \ref{alg:1}. At time $0$, node $i$ initializes its local variables to $x_{i}^0=0$, $\mu_i^0=0$, $\hat{x}_i^0 = 0$ and $\hat{x}_j^0 = 0$ for all $j \in \mathcal{N}_i$. For all times $k$, node $i$ first computes its local primal variable $x_i^{k+1}$ by \eqref{cola-x-node}. The computation of $x_i^{k+1}$ at node $i$ is based on its latest local primal-dual variables $x_i^k$ and $\mu_i^{k}$, the latest broadcast information $\hat{x}_i^k$ of itself and $\hat{x}_j^k$ from its neighbors $j$, as well as the gradient of the local cost function $f_i(x_i)$ at $x_i=x_i^k$. Then $\xi_i^k$, the difference between the newly computed primal variable $x_i^{k+1}$ and the previously transmitted $\hat{x}_i^k$ is calculated and denoted by $\xi_i^{k+1}$. If $\xi_i^{k+1} \geq \tau^{k+1}$, meaning that the difference exceeds the threshold to communicate, node $i$ transmits $x_i^{k+1}$ to neighbors and lets $\hat{x}_i^{k+1} = x_i^{k+1}$. Otherwise, node $i$ does not transmit and lets $\hat{x}_i^{k+1} = \hat{x}_i^k$. On the other hand, if node $i$ receives ${x}_j^{k+1}$ from any neighbor $j$, then it lets $\hat{x}_j^{k+1}={x}_j^{k+1}$. Otherwise, it lets $\hat{x}_j^{k+1} = \hat{x}_j^{k}$. Observe that this communication protocol guarantees that node $i$ and its neighbors store the same state variable $\hat{x}_i^{k+1}$. Finally, the local dual variable $\mu_i^{k+1}$ is updated by \eqref{cola-mu-node}.

\begin{algorithm}[t]
	{\small
		\caption{COLA Run by Node $i$}
		\begin{algorithmic}[1]\label{alg:1}
			\REQUIRE Initialize local variables to $x_{i}^0=0$, $\mu_i^0=0$, $\hat{x}_i^0 = 0$ and $\hat{x}_j^0 = 0$ for all $j \in \mathcal{N}_i$.
			\FOR {times $k = 0, 1, \cdots$}
			\STATE Compute local primal variable $x_i^{k+1}$ by
			\begin{align*}
			x_{i}^{k+1}  = x_{i}^{k} - \frac{1}{2cd_{ii}+\rho}  \left(\nabla f_i(x_{i}^{k})  + c\sum\limits_{j \in \mathcal{N}_i} (\hat{x}_{i}^{k} - \hat{x}_{j}^{k}) + \mu_i^k \right). \nonumber
			\end{align*}
			\STATE Compute $\xi_i^{k+1} = \| \hat{x}_i^{k} - x_i^{k+1}\|$.
			\STATE If $\xi_i^{k+1} \geq \tau^{k+1}$, transmit $x_i^{k+1}$ to
			neighbors and let $\hat{x}_i^{k+1} = x_i^{k+1}$;
			else do not transmit and let $\hat{x}_i^{k+1} = \hat{x}_i^{k}$.
			\STATE If receive ${x}_j^{k+1}$ from any neighbor $j$, let $\hat{x}_j^{k+1}={x}_j^{k+1}$; else let
			$\hat{x}_j^{k+1} = \hat{x}_j^{k}$.
			\STATE Update local dual variable $\mu_i^{k+1}$ as
			\begin{align*}
			\mu_i^{k+1}& = \mu_i^{k} + c\sum_{j\in \mathcal{N}_i} (\hat{x}_{i}^{k+1}-\hat{x}_{j}^{k+1}). \nonumber
			\end{align*}
			\ENDFOR
	\end{algorithmic}}
\end{algorithm}

\begin{remark}

Comparing \eqref{dlm-x-node} and \eqref{admm-mu-node} with \eqref{cola-x-node} and \eqref{cola-mu-node}, we observe that the only difference between DLM and COLA is replacing $c\sum_{j \in \mathcal{N}_i} (x_{i}^{k} - x_{j}^{k})$ by $c\sum_{j \in \mathcal{N}_i} (\hat{x}_{i}^{k} - \hat{x}_{j}^{k})$ in the primal-dual updates. This is not the standard strategy used in the other communication-censored algorithms \cite{Lu2017, Tsianos2013, Chen2016, COCA}, where all the local primal variables $x_i$ are replaced by the state variables $\hat{x}_i$. We customize the communication censoring strategy for COLA and keep the local primal variable $x_i^k$ in $x_{i}^{k} - \frac{1}{2cd_{ii}+\rho} \nabla f_i(x_{i}^{k})$ as it is, because $x_i^k$ has already been available for node $i$, and is more up-to-date than $\hat{x}_i^k$. Recall that the term $x_{i}^{k} - \frac{1}{2cd_{ii}+\rho} \nabla f_i(x_{i}^{k})$ comes from the linearization of $f_i(x_i)$. Intuitively, linearization around $x_i =x_i^k$ leads to faster convergence than linearization around $x_i = \hat{x}_i^k$, which has been validated in our preliminary numerical experiments. On the other hand, we do not change the state variables $\hat{x}_i^k$ by the corresponding local primal variables $x_i^k$ in the term $c\sum_{j \in \mathcal{N}_i} (\hat{x}_{i}^{k} - \hat{x}_{j}^{k})$ in both primal and dual updates. Note that $x_i^k$ and $\hat{x}_i^k$ are not equal when communication censoring happens and the error between them is determined by the censoring threshold $\tau^k$, while the dual update \eqref{cola-mu-node} accumulates all the previous differences between the neighboring state variables $\hat{x}_i^k$ and $\hat{x}_j^k$. Thus, replacing the state variables $\hat{x}_i^k$ therein by the corresponding local primal variables $x_i^k$ shall accumulate the errors, and result in instability or even divergence of the recursion.

\end{remark}

The censoring threshold $\tau^k$ is a critical factor that influences the communication-computation tradeoff of COLA. Setting a large $\tau^k$ prevents less-informative transmissions, and thus reduces the iteration-wise communication cost, though the recursion needs more number of iterations and hence more computation cost to reach a target accuracy. However, a too large $\tau^k$ slows down the convergence speed, which in turn increases both the overall computation and communication costs. Since $\tau^k$ sets an upper bound for the distance between $x_i^k$ and $\hat{x}_i^k$, a small improvement of the local primal variable $x_i^k$ cannot be accepted to the state variable $\hat{x}_i^k$ and diffused to the network. In this sense, the primal variable cannot converge faster than $\tau^k$. We shall give rigorous analysis on this issue in the theoretical analysis.

If we expect to obtain a linear rate of convergence, a choice for the censoring threshold will be
\begin{align}\label{fun:trigger}
\tau^k =\alpha\cdot(\beta)^k,
\end{align}
where $\beta \in (0,1)$ and $\alpha > 0$ are constants. If $\tau^k$ is set as $\alpha\cdot(k)^{-r}$ with $r > 1$, a sublinear rate depending on $r$ will be derived.
A special case is $\tau^k=0$ for all times $k$, meaning that there is no censoring and COLA degenerates to DLM.

\subsection{Tradeoff between Communication and Computation}
\label{sec:algo-tradeoff}

Here we discuss the communication-computation tradeoff in ADMM, DLM, as well as their communication-censored versions, COCA and COLA.

Generally speaking, among the four algorithms, ADMM needs the least number of iterations to reach a target accuracy, but the computation cost of solving subproblems is often remarkable. DLM alleviates the iteration-wise computation cost through linearization, but requires more number of iterations and higher overall communication cost than ADMM.

The communication-censoring strategy in COCA and COLA adjusts the communication-computation tradeoff through tuning the censoring threshold $\tau^k$. As we have discussed in Section \ref{sec:algo-cola}, a larger $\tau^k$ leads to more iterations and thus higher computation cost, but lower iteration-wise communication cost. Regarding the overall communication cost required to reach a target accuracy, there is a phase transition in tuning $\tau^k$. When $\tau^k$ is too large, communication censoring is too often and much more number of iterations is necessary to compensate the information loss, which would deteriorate the overall communication cost.

Though COCA and COLA both adopt the communication-censoring strategy, their application scenarios are different. COCA fits for applications where computation of solving complicated subproblems is not an issue, but communication is the main bottleneck. Examples include distributed resource allocation in a data center network and collaborative target tracking in a radar network. On the contrary, COLA inherits the advantage of light-weight computation from DLM, and further reduces the communication cost on top of it. In this sense, COLA fits for applications where nodes are unable to afford solving complicated subproblems due to hardware or time constraints, such as an IoT network equipped with cheap computation units and a drone network cruising in a fast changing environment.

\begin{figure}
	\begin{center}
		\includegraphics[height=48mm]{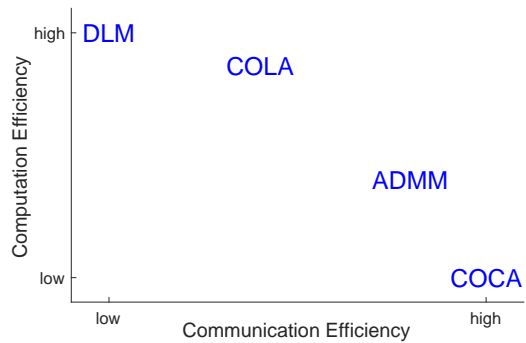}
		\caption{Illustration of the tradeoff between computation efficiency and communication efficiency in ADMM, DLM, COCA and COLA. \label{fig:comp-comm}}
	\end{center}
\end{figure}

An illustration of the tradeoff between computation efficiency and communication efficiency in ADMM, DLM, COCA and COLA is given by Fig. \ref{fig:comp-comm}.

\section{Convergence and Rates of Convergence}
\label{sec:conv}

In this section, we prove that COLA converges to an optimal solution of the convex consensus optimization problem \eqref{eq:obj} under mild conditions. Further, if the local cost functions are strongly convex, COLA converges to the unique optimal solution of \eqref{eq:obj} at a linear or sublinear rate, depending on the choice of the censoring threshold. Section \ref{subsec:pre} provides assumptions and lemmas for the proofs. Section \ref{subsec:conv} analyzes the convergence of COLA, while linear and sublinear rates are established in Section \ref{subsec:linconv}.

\subsection{Preliminaries}
\label{subsec:pre}

We make the following assumptions for the analysis. Assumptions \ref{ass:conn}--\ref{ass:init} are sufficient for proving the convergence of COLA to an optimal solution of \eqref{eq:obj}. Further with Assumption \ref{ass:strong}, COLA is guaranteed to converge to the unique optimal solution of \eqref{eq:obj} at a linear (sublinear) rate when the censoring threshold is linearly (sublinearly) decaying to $0$.

\begin{assumption}
	[\textbf{Network connectivity}]\label{ass:conn} The communication graph $\mathcal{G}=\{\mathcal{V},\mathcal{A}\}$ is bidirectionally connected.
\end{assumption}

\begin{assumption}
	[\textbf{Convexity and differentiability}]\label{ass:convex} The local cost functions $f_i$ are convex and differentiable.
\end{assumption}

\begin{assumption} [\textbf{Lipschitz continuous gradients}]\label{ass:Lip}
	The gradients of the local cost functions $\nabla f_i$ are Lipschitz continuous with constant $M>0$. That is, given any $\tilde{x},\tilde{y} \in \mathcal{R}^p$, $\|\nabla f_i(\tilde{x})-\nabla f_i(\tilde{y})\| \leq M\|\tilde{x}-\tilde{y}\| $ for any $i$.
\end{assumption}

\begin{assumption} [\textbf{Initialization}]\label{ass:init}
	The dual variable $\mu$ of COLA is initialized in the column space of $G_o^T$. That is, there exists a vector $\phi^0\in\mathcal{R}^{mp}$ such that $\mu^0 = G_o^T \phi^0$.
\end{assumption}

\begin{assumption} [\textbf{Strong convexity}]\label{ass:strong}
	The local cost functions $f_i$ are strongly convex with constant $m>0$. That is, given any $\tilde{x},\tilde{y} \in \mathcal{R}^p$, $\langle \nabla	f_i(\tilde{x})-\nabla f_i(\tilde{y}), \tilde{x}-\tilde{y} \rangle \geq m\|\tilde{x}-\tilde{y}\|^2$ for any $i$.
\end{assumption}

Assumptions \ref{ass:conn}, \ref{ass:convex}, \ref{ass:Lip} and \ref{ass:strong} are standard in analysis of decentralized algorithms. The initial condition in Assumption \ref{ass:init} can be easily satisfied, with the simplest choice $\mu^0=0$.

COLA involves a primal sequence $\{x^k\}$ and a dual sequence $\{\mu^k\}$. In the theoretical analysis, we shall construct a triple $(x^k,z^k,\phi^k)$ from the pair $(x^k,\mu^k)$, and prove its convergence to $(x^*,z^*,\phi^*)$, which is optimal to \eqref{eq:obj-ADM}. Here $z^k, z^*, \phi^k, \phi^* \in\mathcal{R}^{mp}$. The next lemma gives the properties of $(x^*,z^*,\phi^*)$.
\begin{lemma}(Lemma 1, \cite{Ling2015-DLM})\label{lemma1}
	Given a primal optimal solution $x^*$ of \eqref{eq:obj-ADM} and $z^* := \frac{1}{2} G_u x^*$, there exist multiple optimal dual variables $[\phi^*;-\phi^*]$ such that every $(x^*,z^*,[\phi^*;-\phi^*])$ is a primal-dual optimal solution of \eqref{eq:obj-ADM}. Among all these optimal dual variables, there exists a unique $[\phi^*;-\phi^*]$ in which $\phi^*$ lies in the column space of $G_o$. Moreover, any primal-dual optimal solution $(x^*,z^*,[\phi^*;-\phi^*])$ satisfies the KKT conditions
	\begin{align}
	\nabla f(x^{*})+ G_o^T\phi^{*}&=0 \label{kkt-1}, \\
	G_o {x}^{*} &= 0 \label{kkt-2}, \\
	\frac{1}{2} G_u {x}^{*} & = z^* \label{kkt-3}.
	\end{align}
\end{lemma}

According to Lemma \ref{lemma1}, it is natural to construct $z^k := \frac{1}{2} G_u x^k \in \mathcal{R}^{mp}$. To construct $\phi^k$, note that under Assumption \ref{ass:init}, $\mu^0$ lies in the column space of $G_o^T$, and by the definition of $L_o=\frac{1}{2}G_o^T G_o$, every $\mu^{k+1}$ in the dual update \eqref{cola-mu} also lies in the column space of $G_o^T$. Thus, there exists a vector $\phi^k\in\mathcal{R}^{mp}$ satisfying $\mu^k = G_o^T \phi^k$ for any $k \geq 0$, such that the recursion of COLA can be rewritten as
\begin{align}
x^{k+1} & = x^k - (2cD+\rho I)^{-1} \left(\nabla f(x^k)  + cL_o\hat{x}^{k} + G_o^T\phi^k \right),\label{update-1}\\
\phi^{k+1} & = \phi^k + \frac{c}{2}G_o\hat{x}^{k+1}.\label{update-2}
\end{align}

Combining \eqref{update-1} and \eqref{update-2} with the KKT conditions \eqref{kkt-1}--\eqref{kkt-3}, the next lemma gives two equations that are cornerstones of the theoretical analysis. To emphasize the error caused by the communication-censoring strategy, we define an error term $E^k := x^k - \hat{x}^k$ therein.

\begin{lemma}\label{lemma2}
	Let $x^*$ and $\phi^*$ be a primal-dual optimal pair of \eqref{eq:obj-ADM}, with $\phi^*$ lying in the column space of $G_o$.
	Then, for all $k \geq 0$, the recursion of COLA satisfies
	\begin{align}
	\nabla f(x^{k}) - \nabla f(x^{*}) =& (cL_u+\rho I)(x^k - x^{k+1})\label{lemma2-1}\\
	& \hspace{-5.5em} - G_o^T(\phi^{k+1}-\phi^*) + cL_o (E^k-E^{k+1}), \notag\\
	\frac{c}{2}G_o(x^{k+1} - x^*) =& \phi^{k+1} -\phi^k+\frac{c}{2}G_oE^{k+1}.\label{lemma2-2}
	\end{align}
\end{lemma}

\begin{proof}
	See Appendix \ref{app_lemma2}.
\end{proof}

The convergence analysis of COLA relies on the following energy function
\begin{equation}\label{eq:energy}
V^k := \frac{\rho}{2}\|x^k-x^*\|^2 + c\|z^k-z^*\|^2 +\frac{1}{c}\|\phi^k-\phi^*\|^2,
\end{equation}
where the auxiliary variables $z^k$ and $\phi^k$ as well as their optimal values $z^*$ and $\phi^*$ are defined above. This energy function also appears in the analysis of DLM, the uncensored version of COLA \cite{Ling2015-DLM}. However, due to the existence of the communication-censoring strategy which introduces an error term in the recursion, the analysis of COLA is significantly different to that of DLM.


\subsection{Convergence}
\label{subsec:conv}

The convergence of COLA is established as follows.

\begin{theorem} \label{theorem:convergence}
	Under Assumptions \ref{ass:conn}--\ref{ass:init}, in COLA we choose the penalty parameter $c>0$ and the linearization parameter $\rho>0$ such that $c\lambda_{\min}(L_u)+\rho > \frac{M}{2}$, and set the censoring threshold $\{\tau^{k}\}$ as a non-increasing non-negative summable sequence	such that $\sum_{k=0}^{\infty} \tau^k < \infty$. Then the primal variable $x^k$	converges to an	optimal solution $x^*$ of (\ref{eq:obj-ADM}).
\end{theorem}

\begin{proof}
	See Appendix \ref{app_convergence}.
\end{proof}

Theorem \ref{theorem:convergence} asserts that COLA converges to an optimal solution of \eqref{eq:obj} under mild conditions and provides guidelines for setting parameters. It is interesting to see that the requirement $c\lambda_{\min}(L_u)+\rho > \frac{M}{2}$ is the same as that in DLM \cite{Ling2015-DLM}. Fixing $\rho$, a network with better connectedness (namely, larger $\lambda_{\min}(L_u)$) allows us to choose a smaller penalty constant $c$. Fixing $c$ and $\lambda_{\min}(L_u)$, the linearization parameter $\rho$ must be large enough to guarantee convergence. Note that $\rho I_p$ approximates the Hessians of the local cost functions $f_i(x_i)$. A large $\rho$ over-approximates the curvature and forces $x_i^{k+1}$ to be close to $x_i^k$, which stabilizes the recursion. On the contrary, a small $\rho$ under-approximates the curvature and allows the local variables to change quickly, at the cost of possible divergence. Fig. \ref{fig:rho} illustrates the impact of $\rho$. Regarding the censoring threshold $\tau^k$, we require it to be summable. Intuitively, $\tau^k$ determines the maximal error introduced to the primal update. When this error is controllable, the convergence of COLA is guaranteed

\begin{figure}
	\begin{center}
		\includegraphics[height=48mm]{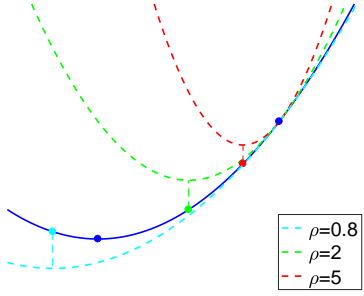}
		\caption{An illustration of choosing different approximation parameters $\rho$. In this situation, at the top-right point, we approximate the original cost function (in blue) by the dashed lines (in cyan, green and red). When $\rho=2$ and $\rho=5$ that are both larger than the accurate second derivative, the updates are conservative and go to the green and red points, respectively. When $\rho=0.8$ that is smaller than the accurate second derivative, the update is aggressive and goes to the cyan point. \label{fig:rho}}
	\end{center}
\end{figure}

\subsection{Rates of Convergence}
\label{subsec:linconv}

In Section \ref{subsec:conv}, we have shown that COLA requires $\{\tau^k\}$ to be summable so as to guarantee convergence. Below, we shall prove that convergence rate of COLA also depends on convergence rate of $\{\tau^k\}$. In addition to Assumptions \ref{ass:conn}--\ref{ass:init}, we need the local cost functions to be strongly convex, as stated in Assumption \ref{ass:strong}. In this circumstance, COLA converges to the unique optimal solution of \eqref{eq:obj} at a linear (sublinear) rate when $\{\tau^k\}$ is linearly (sublinearly) decaying.

\begin{theorem} \label{theorem:linear}
	Under Assumptions \ref{ass:conn}--\ref{ass:strong}, in COLA we choose the penalty parameter $c>0$ and the linearization parameter $\rho>\frac{M^2}{2m}$, and set the censoring threshold $\tau^{k} = \alpha\cdot(\beta)^k$ with $\alpha>0$ and $\beta\in (0,1)$. Then there exists a positive constant $\delta>0$ such that the primal variable $x^k$ converges to the unique optimal solution $x^*$ of (\ref{eq:obj-ADM}) at a global linear rate $\mathcal{O}((1+\delta)^{-\frac{k}{2}})$.
\end{theorem}

\begin{proof}
	See Appendix \ref{app_linear}.
\end{proof}

As shown in \eqref{eq:delta-bound} in the proof of Theorem \ref{theorem:linear}, the constant $\delta$ depends on the algorithm parameters $c$, $\rho$ and $\beta$, the network topology parameterized by $\tilde{\sigma}_{\min}(G_o)$ and $\sigma_{\max}(G_u)$, and the properties of the local cost functions parameterized by $M$ and $m$. Define the condition numbers of cost functions and graph as $\kappa_f=\frac Mm$ and $\kappa_G=\frac{\sigma_{\max}(G_u)}{\tilde{\sigma}_{\min}(G_o)}$, respectively. The following corollary shows clearer that, by properly setting $c$ and $\rho$, how the constant $\delta$ is determined by $\kappa_f$, $\kappa_G$ and $\beta$.

\begin{corollary}\label{cor:delta}
Under Assumptions \ref{ass:conn}--\ref{ass:strong}, in COLA we choose $c=\frac{8M}{\sigma_{\max}(G_u)\tilde{\sigma}_{\min}(G_o)}$ and $\rho=M \kappa_f$. Then the global linear rate $\mathcal{O}((1+\delta)^{-\frac{k}{2}})$ satisfies
\begin{align}\label{eq:cor}
\delta\le\min\left\{\frac1{8\kappa_G^2},\frac1{2\kappa_f^2+16\kappa_f \kappa_G},\frac{\kappa_f}{12\kappa_G+6\kappa_f^2\kappa_G},\frac1{\beta^2}-1\right\}.
\end{align}
\end{corollary}

In \eqref{eq:cor}, the terms $\frac1{8\kappa_G^2}$, $\frac1{2\kappa_f^2+16\kappa_f \kappa_G}$ and $\frac{\kappa_f}{12\kappa_G+6\kappa_f^2\kappa_G}$ are monotonically decreasing when either $\kappa_f$ or $\kappa_G$ increases, meaning that the convergence is slower when the cost functions are worse-conditioned and/or the network is less-connected. In addition, $\delta$ is bounded by $\frac1{\beta^2}-1$. Therefore, \eqref{eq:cor} shows that among all $\beta$ that do not affect the convergence rate, the one satisfying $\min\ \{ \frac1{8\kappa_G^2},\frac1{2\kappa_f^2+16\kappa_f \kappa_G},\frac{\kappa_f}{12\kappa_G+6\kappa_f^2\kappa_G} \}$ $= \frac1{\beta^2}-1$ achieves largest communication reduction per iteration.

\begin{theorem} \label{theorem:sublin}
	Under Assumptions \ref{ass:conn}--\ref{ass:strong}, in COLA we choose the penalty parameter $c>0$ and the linearization parameter $\rho>\frac{M^2}{2m}$, and set the censoring threshold $\tau^{k} = \alpha\cdot (k)^{-r}$ with $\alpha>0$ and $r>1$. Then there exists a finite time index $k_0$ such that the distance between the primal variable $x^k$ and the unique optimal solution $x^*$ of (\ref{eq:obj-ADM}) is upper-bounded by a sequence decaying sublinearly to $0$ at a rate of $\mathcal{O}((k)^{-\frac{q}{2}})$, where $q \in (0, 2r-1)$, when $k \geq k_0$.
\end{theorem}


\begin{proof}
	See Appendix \ref{app_sublin}.
\end{proof}

Theorems \ref{theorem:linear} and \ref{theorem:sublin} indicate that, to achieve linear (sublinear) convergence, we have to impose stronger requirements on the parameters. The sequence of censoring threshold should be not only summable, but also linearly (sublinearly) decaying. The parameters $c$ and $\rho$ should be larger, too. Note that because $M \geq m$, $\rho>\frac{M^2}{2m} \geq \frac{M}{2}$ and consequently $c\lambda_{\min}(L_u)+\rho > \frac{M}{2}$, which is required in Theorem \ref{theorem:convergence}.


According to the upper bound of $\delta$ given in \eqref{eq:delta-bound}, the linear rate of $x^k$ reaching $x^*$ (namely, $\mathcal{O}((1+\delta)^{-k/2})$) must be slower than the linear rate of $\tau^k$ decaying to $0$ (namely, $\mathcal{O}({\beta}^k)$). From Theorem \ref{theorem:sublin}, one can also see that the sublinear rate of $x^k$ reaching $x^*$ (namely, $\mathcal{O}((k)^{-\frac{q}{2}})$ where $q \in (0, 2r-1)$) must be slower than the sublinear rate of $\tau^k$ decaying to $0$ (namely, $\mathcal{O}((k)^{-r})$). Therefore, in both the linear and the sublinear cases, the sequence of censoring threshold $\tau^k$ bounds the convergence rate of $x^k$ to $x^*$. This makes sense because $\tau^k$ means the maximal error allowed to enter the recursion of $x^k$ due to communication censoring.


\begin{remark}
	Though COLA is devised from DLM, the error caused by the communication-censoring strategy makes its analysis different to that of DLM. The analysis of COLA is also different to that of COCA, the censored version of ADMM. The reason is that COLA updates $x^k$ by gradient descent steps, while COCA updates $x^k$ by solving optimization subproblems. This is analogous to the difference in the proofs of DLM and ADMM. In addition to the difference in the proof techniques, we also establish the sublinear convergence of COLA, which is absent in the analysis of DLM and COCA.
\end{remark}

\begin{remark}
	When the censoring threshold $\tau^k$ is set to $0$, COLA degenerates to DLM. Intuitively, the convergence rate of COLA is no faster than that of DLM due to the introduction of the communication-censoring strategy. This is also observed from, for example, the linear convergence constant $\delta$ in Corollary \ref{cor:delta}. Nevertheless, the slower convergence in terms of the number of iterations is acceptable, since COLA effectively reduces the iteration-wise communication cost. We shall demonstrate with numerical experiments that COLA can reduce the overall communication cost comparing to DLM.
\end{remark}


\section{Numerical Experiments}
\label{sec:nume}

This section provides numerical experiments to demonstrate the satisfactory communication-computation tradeoff of COLA. In particular, we shall show that COLA inherits the advantage of cheap computation from its uncensored counterpart DLM \cite{Ling2015-DLM, Chang2015}, but significantly reduces the overall communication cost. Beyond DLM, we compare COLA with the classical ADMM \cite{Shi2014-ADMM} and its censored version COCA \cite{COCA}, both of which do not use the linearization technique and are not computation-efficient. We also compare with the event-triggered sub-gradient descent (ETSD) algorithm \cite{Lu2017}, which is a primal domain first-order method but much slower than COLA in terms of convergence speed. We consider two decentralized consensus optimization problems, least squares in \ref{subsec:ls} and logistic regression in \ref{subsec:lr}. The cost functions are both smooth, but the latter is not strongly convex. For ADMM and COCA, subproblems in least squares have explicit solutions, while those in logistic regression needs computationally demanding inner loops. We use the accuracy of the primal variable as the performance metric, defined by $\|x^k-x^*\|^2/\|x^0-x^*\|^2$. Logistic regression may have multiple optimal solutions, among which we choose the one closest to the limit of iterate as $x^*$. The computation cost is evaluated by time spent to reach a target accuracy, and the communication cost is defined as the accumulated number of broadcast messages. The simulations are carried out on a laptop with an Intel I7 processor and 8GB memory, programmed with Matlab R2017a in macOS Sierra.

\subsection{Decentralized Least Squares}
\label{subsec:ls}

The local cost function in the decentralized least squares problem is $f_i(\tilde{x}) = \frac{1}{2} \|A_{(i)}\tilde{x}-y_{(i)}\|^2_2$, with $A_{(i)}\in \mathcal{R}^{p \times p}$ and $y_{(i)}\in \mathcal{R}^p$ being private for node $i$. Thus, the primal update of node $i$ at time $k$ in COLA is
\begin{align*}
x_{i}^{k+1}  =& x_{i}^{k} - (2cd_{ii}+\rho)^{-1}  \bigg[A_{(i)}^T(A_{(i)}x_{i}^{k}-y_{(i)})\\
& \hspace{8.5em} + c\sum\limits_{j \in \mathcal{N}_i} (\hat{x}_{i}^{k} - \hat{x}_{j}^{k}) + \mu_i^k \bigg].
\end{align*}
Note that node $i$ can compute $(2cd_{ii}+\rho)^{-1}$ in advance to accelerate the computation. In the experiments, entries of $A_{(i)}$ and $b_{(i)}$ are independently and identically sampled from the uniform distribution within $[0,1]$. Then we let $y_{(i)}=A_{(i)}b_{(i)}$. We set the network size as $n=50$ and the dimension of the local variables as $p=3$.

First, we compare four algorithms, COLA, DLM, COCA and ADMM, over four network topologies: line, random, star and complete, as shown in Figs. \ref{fig:ls-line}--\ref{fig:ls-c}. In the random network, $10\%$ of all possible bidirectional edges are randomly chosen to be connected. The accuracies are compared with respect to the number of iterations and the cumulative communication cost. The parameters $c$ and $\rho$ are tuned to be the best for the uncensored algorithms DLM and ADMM, and kept the same in their censored counterparts, respectively. We use the linear censoring threshold in the form of $\tau^k = \alpha\cdot(\beta)^k$, where the parameters $\alpha$ and $\beta$ are hand-tuned in COLA and COCA so as to achieve the best communication efficiency. Taking the random network as an example, we choose $c=0.45$, $\rho=1.1$ in DLM and $\alpha=0.7$, $\beta=0.94$ in COLA, while $c=0.35$ in ADMM and $\alpha=0.9$, $\beta=0.92$ in COCA.

In all the networks, the two censored methods COLA and COCA require more iterations to reach the target accuracy than their uncensored counterparts due to the error caused by censoring, but the saving in communication is remarkable. Compared to DLM and given a target accuracy of $10^{-8}$, COLA saves $\sim1/2$ communication costs in the line and random networks, and $\sim 1/3$ in the star and complete networks. The required number of iterations in the line network is much more than those in the other networks, since the connectedness of the line network is the worst. In better connected networks such as star and complete, variable updating is often informative, such that the deterioration of convergence speed caused by skipping transmissions becomes more noticeable, yet communication per iteration is still saved by censoring.

\begin{figure}
	\begin{center}
		\includegraphics[height=48mm,trim=0 20 0 20]{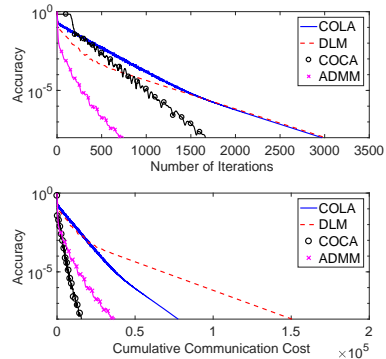}
		\caption{Performance over line network for decentralized least squares. \label{fig:ls-line}}
	\end{center}
\end{figure}

\begin{figure}
	\begin{center}
		\includegraphics[height=48mm,trim=0 20 0 20]{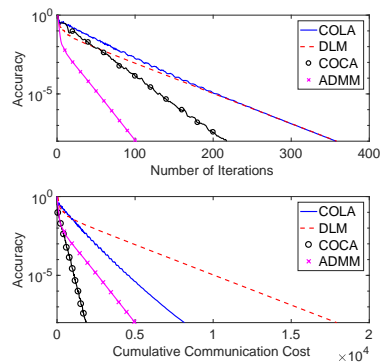}
		\caption{Performance over random network for decentralized least squares.\label{fig:ls-r}}
	\end{center}
\end{figure}

\begin{figure}
	\begin{center}
		\includegraphics[height=48mm,trim=0 20 0 20]{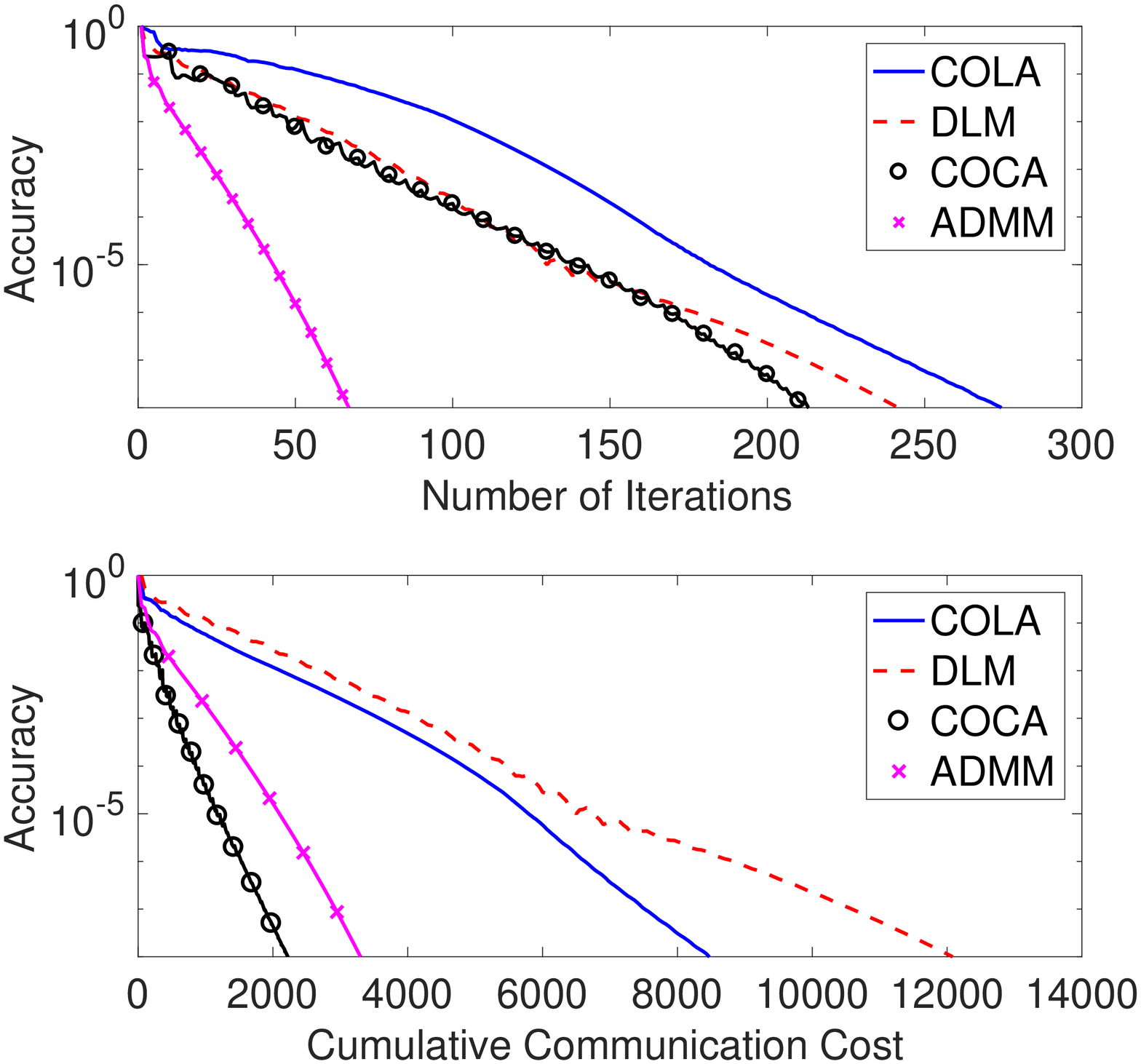}
		\caption{Performance over star network for decentralized least squares.\label{fig:ls-star}}
	\end{center}
\end{figure}

\begin{figure}
	\begin{center}
		\includegraphics[height=48mm,trim=0 20 0 20]{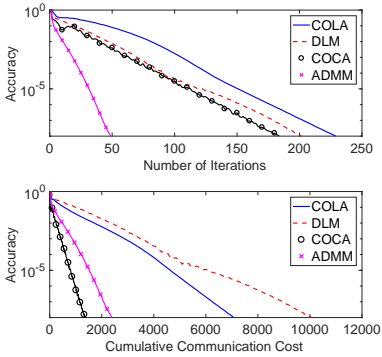}
		\caption{Performance over complete network for decentralized least squares.\label{fig:ls-c}}
	\end{center}
\end{figure}

We study the influence of communication censoring over the random network. The censoring pattern of the first $200$ iterations is shown in Fig. \ref{fig:ls-t1}. The horizontal axis is the number of iterations, and the vertical axis is the node index. A white dot means that the node broadcasts at the time, while a black dot means that the node is silent. Observe that communication censoring happens uniformly, namely, the frequency of communication censoring does not change too much along the optimization process. In addition,  the nodes have similar communication costs eventually. On average, every node broadcasts $0.35 \sim 0.45$ message per time.

Next, we compare the choice of the censoring threshold in COLA over the random network. We compare four censoring thresholds, the linear sequences $\tau^k = \alpha\cdot(\beta)^k$ with $\alpha=0.7$ while $\beta= 0.93$, $0.95$ and $0.97$, as well as the sublinear sequence $\tau^k = \alpha\cdot(k)^{-r}$ with $\alpha=1000$ and $r=2.5$. The parameters $c$ and $\rho$ remain the same. As shown in Fig. \ref{fig:ls-beta}, the linear censoring thresholds outperforms the sublinear censoring threshold, in terms of both communication and computation. The reason is that the sublinear rate of the threshold limits the convergence rate of COLA, as we have theoretically analyzed in Section \ref{sec:conv}. Regarding the
different choices of the linear rate, we observe that a smaller $\beta$ needs less number of iterations to reach a target accuracy, since it leads to faster decay of the censoring threshold, and thus less communication censoring per iteration. In contrast, with a larger $\beta$, we need more number of iterations and less communication cost per iteration.
Therefore, a moderate $\beta$, such as $\beta=0.95$ in this case, is preferred.

\begin{figure}
	\begin{center}
		\includegraphics[height=35mm,trim=120 20 0 20]{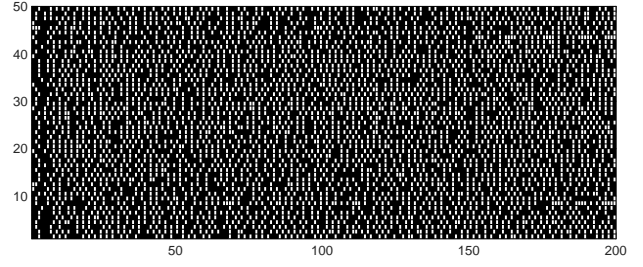}
		\caption{Censoring pattern of the first 200 iterations of COLA over random network for decentralized least squares.
			The horizontal axis is the number of iterations, and the vertical axis is the index of node.
			A dark dot represents that the node is censored at that time.\label{fig:ls-t1}}
	\end{center}
\end{figure}

\begin{figure}
	\begin{center}
		\includegraphics[height=48mm,trim=0 20 0 20]{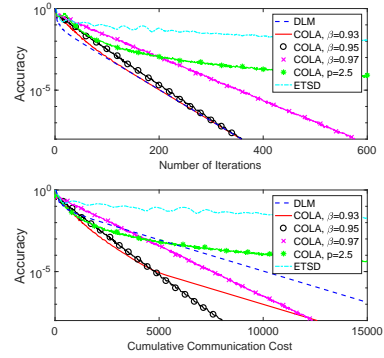}
		\caption{Performance of DLM, ETSD and COLA with different censoring thresholds over random network for decentralized least squares. In COLA, with the linear censoring thresholds $\tau^k = \alpha\cdot(\beta)^k$, we fix $\alpha=0.7$ and choose different $\beta$. With the sublinear censoring threshold $\tau^k = \alpha(k)^{-r}$, we choose $\alpha=1000$ and $r=2.5$. \label{fig:ls-beta}}
	\end{center}
\end{figure}

In Fig. \ref{fig:ls-beta}, we also compare COLA with ETSD, a communication-censored primal domain first-order method. ETSD adopts the Metropolis-Hastings rule to design its mixing matrix. It uses a linear censoring threshold $\alpha\cdot(\beta)^k$ and a sublinear step size $O((k)^{-\frac{2}{3}})$, where the parameters are all hand-tuned to achieve the best communication efficiency. From Fig. \ref{fig:ls-beta}, we observe that ETSD requires much more number of iterations and communication cost to reach a target accuracy comparing to COLA. The main reason of the unsatisfactory performance of ETSD is the diminishing step size, which is used to guarantee exact convergence. Similar performance gap can be observed in comparing the uncensored algorithms, sub-gradient descent and DLM.

\subsection{Decentralized Logistic Regression}
\label{subsec:lr}

In the decentralized logistic regression problem, the local cost function of node $i$ is
\begin{align*}
f_i(\tilde{x}) = \frac{1}{l_i} \sum_{l=1}^{l_i}  \ln \bigg(1+\exp(- y_{(i)l} q_{(i)l}^T \tilde{x} ) \bigg), \nonumber
\end{align*}
where $q_{(i)l} \in \mathcal{R}^p$ is the $l$th column of a matrix $Q_{(i)} \in\mathcal{R}^{p \times l_i}$, $y_{(i)l} \in \{-1, +1\}$ is the $l$th element of a binary vector $y_{(i)} \in \mathcal{R}^{l_i}$, and $l_i$ is the number of samples held by node $i$. The primal update of node $i$ at time $k$ in COLA is
\begin{align*}
x_{i}^{k+1}  = &x_{i}^{k} - \frac{1}{2cd_{ii}+\rho}
\bigg[ -\frac{1}{l_i} \sum_{l=1}^{l_i}
\frac{ y_{(i)l} \exp(- y_{(i)l} q_{(i)l}^T x_{i}^{k} ) q_{(i)l} }{1+\exp(- y_{(i)l} q_{(i)l}^T x_{i}^{k} )}\\
&\hspace{6.8em}+ c\sum\limits_{j \in \mathcal{N}_i} (\hat{x}_{i}^{k} - \hat{x}_{j}^{k}) + \mu_i^k \bigg],
\end{align*}
while the primal updates of ADMM and COCA have no explicit solutions. Therefore, we solve the subproblems therein by a gradient descent inner loop, which terminates when the $\ell_2$ norm of the gradient is less than $10^{-8}$.

\begin{figure}
	\centering
	\includegraphics[height=48mm,trim=0 20 0 20]{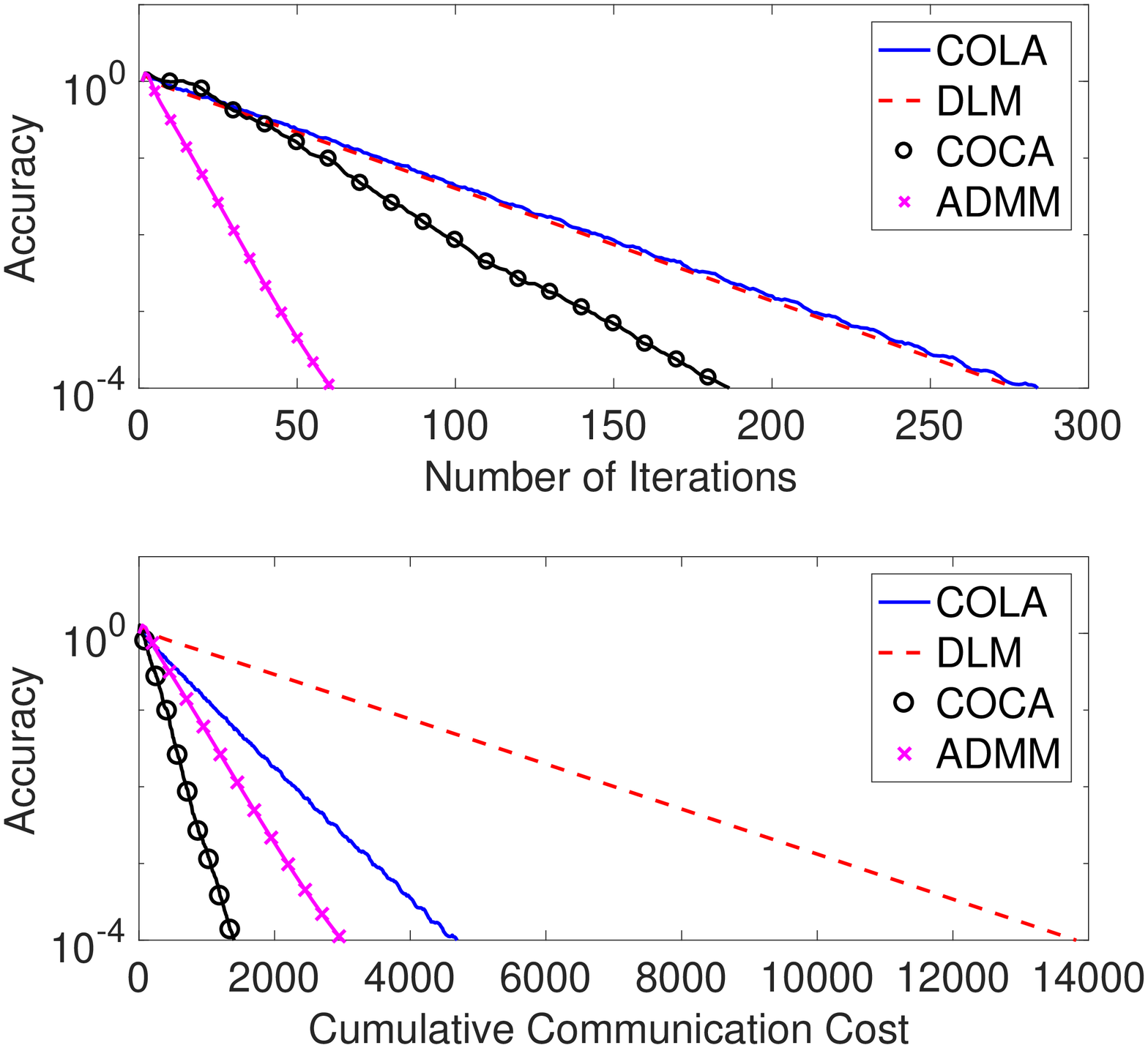}
	\caption{Performance over random network with $50$ nodes for decentralized logistic regression.\label{fig:lr-50}}
\end{figure}

\begin{figure}
	\centering
	\includegraphics[height=48mm,trim=0 20 0 20]{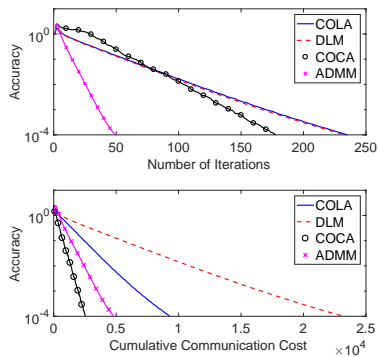}
	\caption{Performance over random network with $100$ nodes for decentralized logistic regression.\label{fig:lr-100}}
\end{figure}

We conduct simulations over two random networks with $n=50$ and $n=100$ nodes, in both of which $10\%$ of all possible bidirectional edges are randomly chosen to be connected. The dimension of local variables is $p=3$. The numbers of samples held by the nodes are i.i.d. and uniformly chosen from integers within $[1, 10]$. Entries of the first two rows of $Q_{(i)}$ follow the i.i.d. discrete uniform distribution on the set $\{0.1w\}$, $w=1,\cdots,10$, while entries of the last row are all set as $1$. Entries of $y_{(i)}$ are i.i.d. and follow the uniform distribution on $\{-1, 1\}$. As we have done in Section \ref{subsec:ls}, $c$ in ADMM is tuned to achieve the fastest convergence, and is also used for COCA; $c$ and $\rho$ in DLM are tuned to achieve the fastest convergence, and are also used for COLA. The censoring threshold in both COCA and COLA is set as $\tau^k = \alpha\cdot(\beta)^k$, with parameters $\alpha$ and $\beta$ hand-tuned to obtain the best communication efficiency.

As depicted in Figs. \ref{fig:lr-50} and \ref{fig:lr-100}, the four algorithms behave similarly to those in the least squares problem (Figs. \ref{fig:ls-line}--\ref{fig:ls-c}).
COLA saves nearly $2/3$ communication cost with few more iterations compared to DLM. To demonstrate its computation efficiency, we also show the CPU time for the four algorithms to reach target accuracies $10^{-4}$ and $10^{-5}$ in Table \ref{table:lr-time}. Notice that the two linearized algorithms, COLA and DLM, compute much faster than COCA and ADMM. The time spent by COLA in both networks is slightly more than that of DLM due to the communication censoring operations.

\begin{table}
	\centering
	\begin{tabular}{|c|c|c|c|c|c|}
		\hline
		$n$& Accuracy& COLA& DLM& COCA& ADMM\\
		\hline
		50& $10^{-4}$& 1.076s & 0.971s & 30.119s & 9.629s\\
		50& $10^{-5}$& 1.126s & 1.055s & 37.198s & 12.467s\\
		100& $10^{-4}$& 1.466s & 1.320s & 37.488s & 11.175s\\
		100& $10^{-5}$& 1.879s & 1.721s & 45.302s & 15.797s\\
		\hline
	\end{tabular}
	\caption{The time spent of the four algorithms in two networks with different numbers of nodes $n$ and target accuracies.\label{table:lr-time}}
\end{table}

\section{Conclusion}
\label{sec:con}

In this paper, we propose COLA, a communication- and computation-efficient decentralized consensus optimization algorithm. Compared to the classical ADMM, COLA uses the linearization technique to reduce the iteration-wise computation cost, and fits for networks where only light-weight computation is affordable. To compensate the sacrifice in the convergence speed, which is caused by the linearization step and results in low communication efficiency, COLA further introduces the communication-censoring strategy to prevent a node from transmitting its ``less-informative'' local variable to neighbors. We establish convergence and rates of convergence for COLA, and demonstrate the computation-communication tradeoff with numerical experiments. Our future work is to apply the linearization and communication censoring techniques to decentralized optimization applications in dynamic, online and stochastic environments.

\newpage

\appendices
\section{Proof of Lemma \ref{lemma2}}
\label{app_lemma2}

\begin{proof}
From \eqref{update-2}, it holds that
\begin{align} \label{lem2-0}
\phi^{k+1} -\phi^k &= \frac{c}{2}G_o\hat{x}^{k+1} \overset{\eqref{kkt-2}}{=} \frac{c}{2}G_o(\hat{x}^{k+1} - x^*) \\
&=\frac{c}{2}G_o(x^{k+1} - E^{k+1} - x^*), \notag
\end{align}
where the last equality uses the definition $E^{k+1}=x^{k+1}-\hat{x}^{k+1}$. Rearranging terms in \eqref{lem2-0} yields \eqref{lemma2-1}.

Also rearranging terms in \eqref{update-1} to place $\nabla f(x^{k})$ at the left side, we have
\begin{equation}
\nabla f(x^k) = (2cD+\rho I)(x^k - x^{k+1}) - cL_o\hat{x}^{k} - G_o^T\phi^k. \label{lem2-1}
\end{equation}
Subtracting \eqref{lem2-1} with (\ref{kkt-1}) and noticing the definitions of $D=\frac{1}{2}(L_o+L_u)$ and $L_o = \frac{1}{2}G_o^T G_o$, we have
\begin{align*}
  & \nabla f(x^{k}) - \nabla f(x^{*})  \\
= & (2cD+\rho I)(x^k - x^{k+1}) - cL_o\hat{x}^{k} -
G_o^T(\phi^{k}-\phi^*) \\
= & (cL_u+\rho I)(x^k - x^{k+1}) +cL_o(x^k - x^{k+1}) - cL_o \hat{x}^k \\
&- G_o^T(\phi^{k+1}-\phi^*)+G_o^T(\phi^{k+1}-\phi^k)\\
\overset{\eqref{lem2-0}}{=} & (cL_u+\rho I)(x^k - x^{k+1}) + cL_o(x^k - x^{k+1}) - cL_o \hat{x}^k  \\
&- G_o^T(\phi^{k+1}-\phi^*)
+cL_o (x^{k+1} - E^{k+1} - x^*) \\
\overset{\eqref{kkt-2}}{=} & (cL_u+\rho I)(x^k - x^{k+1}) - G_o^T(\phi^{k+1}-\phi^*) \\
&+ cL_o (E^{k} - E^{k+1}),
\end{align*}
which completes the proof.	
\end{proof}

\section{Proof of Theorem \ref{theorem:convergence}}
\label{app_convergence}

\begin{proof} Throughout the proof, we assume $\tau^0 > 0$. When $\tau^0 = 0$ such that all $\tau^k = 0$ and COLA degenerates to DLM, since the value of $\tau^0$ does not affect the operation of COLA, we can simply set $\tau^0$ as any positive constant.
	
\textbf{Step 1.} The proof in this step is analogous to the proof of Lemma 3 in \cite{Ling2015-DLM}, but more complicated due to the existence of censoring error. From Assumptions \ref{ass:convex} and \ref{ass:Lip}, the gradients of the convex local cost functions $\nabla f_i$ are Lipschitz continuous with constant $M>0$. Thus, we have
\begin{align}\label{strong-cor}
&\frac{1}{M}\|\nabla f(x^{k}) - \nabla f(x^{*})\|^2 \\
\leq& \langle \nabla f(x^{k}) - \nabla f(x^{*}) , x^k - x^*\rangle \notag\\
=& \langle \nabla f(x^{k}) - \nabla f(x^{*}) , x^{k+1} - x^*\rangle \notag\\
 &+ \langle \nabla f(x^{k}) - \nabla f(x^{*}) , x^k - x^{k+1}\rangle. \notag
\end{align}

For the second term at the right-hand side of \eqref{strong-cor}, we choose an upper bound
\begin{align}\label{th1-step1-0}
&\langle \nabla f(x^{k}) - \nabla f(x^{*}) , x^k - x^{k+1}\rangle\\
\leq &\frac{1}{M}\|\nabla f(x^{k}) - \nabla f(x^{*})\|^2
+ \frac{M}{4}\|x^k - x^{k+1}\|^2. \notag
\end{align}

To establish an upper bound for the first term at the right-hand side of \eqref{strong-cor}, we use \eqref{lemma2-1} in Lemma \ref{lemma2} to rewrite it as
\begin{align}\label{th1-step1-1}
& \langle \nabla f(x^{k}) - \nabla f(x^{*}) , x^{k+1} - x^*\rangle \\
= & \langle (cL_u+\rho I)(x^k-x^{k+1}) , x^{k+1} - x^*\rangle \notag\\
&- \langle G_o^T(\phi^{k+1}-\phi^*) , x^{k+1} - x^*\rangle \notag\\
&+ \langle cL_o (E^{k}-E^{k+1}) , x^{k+1} - x^*\rangle. \nonumber
\end{align}
We shall handle the terms at the right-hand side of \eqref{th1-step1-1} one by one. The first one satisfies
\begin{align}\label{th1-step1-1-temp-xxxx-1}
&\langle (cL_u+\rho I)(x^k-x^{k+1}), x^{k+1} - x^*\rangle  \\
= &  \frac{c}{2} \langle G_u (x^k-x^{k+1}), G_u(x^{k+1} - x^*)\rangle \notag\\
&+ \rho \langle x^k-x^{k+1}, x^{k+1} - x^*\rangle \notag\\
= & 2c\langle z^{k} - z^{k+1}, z^{k+1} - z^{*}\rangle + \rho\langle x^k-x^{k+1} , x^{k+1} - x^*\rangle, \nonumber
\end{align}
which uses the definitions $L_u = \frac{1}{2}G_u^T G_u$, $z^k = \frac{1}{2} G_u x^k$ and $z^* = \frac{1}{2} G_u x^*$. The second one satisfies
\begin{align}\label{th1-step1-1-temp-xxxx-2}
&-\langle G_o^T(\phi^{k+1}-\phi^*) , x^{k+1} - x^*\rangle \\
=& -\langle \phi^{k+1}-\phi^* , G_o(x^{k+1} - x^*)\rangle \notag\\
\overset{\text{\eqref{lemma2-2}}}{=} &
\frac{2}{c}\langle \phi^{k+1} - \phi^{*},\phi^{k} - \phi^{k+1}\rangle
-\langle \phi^{k+1}-\phi^* , G_oE^{k+1}\rangle. \nonumber
\end{align}
The third one satisfies
\begin{align}\label{th1-step1-1-temp-xxxx-3}
&\langle cL_o  (E^{k}-E^{k+1}) , x^{k+1} - x^*\rangle \\
=& \frac{c}{2} \langle G_o (E^{k}-E^{k+1}), G_o(x^{k+1} - x^*)\rangle \notag\\
\overset{\text{\eqref{lemma2-2}}}{=} & \langle G_o(E^{k}-E^{k+1}) , \phi^{k+1} - \phi^k\rangle  \notag\\
&+ \frac{c}{2} \langle G_o(E^{k}-E^{k+1}), G_oE^{k+1}\rangle, \nonumber
\end{align}
which uses the definition $L_o = \frac{1}{2}G_o^T G_o$. Summing up \eqref{th1-step1-1-temp-xxxx-1}, \eqref{th1-step1-1-temp-xxxx-2} and \eqref{th1-step1-1-temp-xxxx-3}, applying the equality $2\langle v_a-v_b, v_b-v_c \rangle = \|v_a-v_c\|^2 - \|v_a-v_b\|^2 - \|v_b-v_c\|^2$ that holds for any vectors $v_a$, $v_b$ and $v_c$ to $\langle z^{k} - z^{k+1}, z^{k+1} - z^{*}\rangle$, $\langle x^k-x^{k+1} , x^{k+1} - x^*\rangle$ and $\langle \phi^{k+1} - \phi^{*},\phi^{k+1} - \phi^{k}\rangle$, and then reorganizing terms, we can rewrite \eqref{th1-step1-1} as
\begin{align}
& \langle \nabla f(x^{k}) - \nabla f(x^{*}) , x^{k+1} - x^*\rangle \label{th1-step1-2} \\
=&  c( \|z^k - z^*\|^2 - \|z^k - z^{k+1}\|^2 - \|z^{k+1} - z^*\|^2) \notag \\
&+\frac{\rho}{2} ( \|x^k - x^*\|^2 - \|x^k - x^{k+1}\|^2 - \|x^{k+1} - x^*\|^2)\notag\\
&+\frac{1}{c} ( \|\phi^k - \phi^*\|^2 - \|\phi^k - \phi^{k+1}\|^2 - \|\phi^{k+1} - \phi^*\|^2) \notag\\
&-\langle \phi^{k+1}-\phi^* , G_oE^{k+1}\rangle
+ \langle G_o(E^{k}-E^{k+1}) , \phi^{k+1} - \phi^k\rangle \notag\\
&+ \frac{c}{2} \langle G_o(E^{k}-E^{k+1}) , G_oE^{k+1}\rangle \notag\\
\overset{\text{\eqref{eq:energy}}}{=} & V^k -V^{k+1} \notag\\
&- c\|z^k - z^{k+1}\|^2
- \frac{\rho}{2}\|x^k - x^{k+1}\|^2
- \frac{1}{c}\|\phi^k-\phi^{k+1}\|^2 \notag\\
&-\langle G_oE^{k+1} , \phi^{k}-\phi^* \rangle
+ \langle  G_o(E^{k} - 2E^{k+1}), \phi^{k+1} - \phi^k \rangle \notag\\
&+ \frac{c}{2} \langle G_o(E^{k}-E^{k+1}), G_oE^{k+1}\rangle.\notag
\end{align}

For the term $-\langle G_oE^{k+1}, \phi^{k}-\phi^* \rangle$, we observe that
\begin{align} \label{th1-step1-2-temp-yyyy-1}
&-\langle G_oE^{k+1}, \phi^{k}-\phi^* \rangle \\
\leq& \frac{c_1}{2} \|G_oE^{k+1}\| \|\phi^{k}-\phi^*\|^2 + \frac{1}{2c_1} \|G_oE^{k+1}\| \notag\\
\leq& \frac{c_1 \sigma_{\max}(G_o) \|E^{k+1}\|}{2} \|\phi^{k}-\phi^*\|^2 + \frac{ \sigma_{\max}(G_o)}{2c_1} \|E^{k+1}\|, \nonumber
\end{align}
where $c_1>0$ is any positive constant. Similarly, for $\langle G_o(E^{k} - 2E^{k+1}), \phi^{k+1} - \phi^k \rangle$, it holds
\begin{align} \label{th1-step1-2-temp-yyyy-2}
&\langle G_o(E^{k} - 2E^{k+1}), \phi^{k+1} - \phi^k \rangle \\
\leq& \frac{c_2}{2} \|G_o(E^{k} - 2E^{k+1})\| \|\phi^{k+1} - \phi^k\|^2 \notag\\
&+ \frac{1}{2c_2} \|G_o(E^{k} - 2E^{k+1})\| \notag\\
\leq& \frac{c_2 \sigma_{\max}(G_o)(\|E^{k}\|+2\|E^{k+1}\|) }{2} \|\phi^{k+1} - \phi^k\|^2 \notag\\
&+ \frac{ \sigma_{\max}(G_o)}{2c_2}(\|E^{k}\|+2\|E^{k+1}\|), \nonumber
\end{align}
where $c_2>0$ is any positive constant.
For $\frac{c}{2} \langle G_o(E^{k}-E^{k+1}), G_oE^{k+1}\rangle$, we have
\begin{align} \label{th1-step1-2-temp-yyyy-3}
&\frac{c}{2} \langle G_o(E^{k}-E^{k+1}), G_oE^{k+1}\rangle \\
\leq& \frac{c}{2} \langle G_o E^{k}, G_oE^{k+1}\rangle \notag\\
\leq& \frac{c}{4} \|G_o E^{k}\|^2 + \|G_oE^{k+1}\|^2 \notag \\
\leq& \frac{c}{4} \sigma_{\max}^2(G_o) \|E^{k}\|^2 + \frac{c}{4} \sigma_{\max}^2(G_o) \|E^{k+1}\|^2. \notag
\end{align}
Using \eqref{th1-step1-2-temp-yyyy-1}, \eqref{th1-step1-2-temp-yyyy-2} and \eqref{th1-step1-2-temp-yyyy-3} to rewrite \eqref{th1-step1-2} followed by substituting the result and \eqref{th1-step1-0} into \eqref{strong-cor}, we obtain
\begin{align}
\hspace{-1em}& V^k -V^{k+1}  \label{th1-step1-3}\\
\hspace{-1em} - &c\|z^k - z^{k+1}\|^2 - \left( \frac{\rho}{2} -\frac{M}{4} \right) \|x^k - x^{k+1}\|^2 - \frac{1}{c}\|\phi^k-\phi^{k+1}\|^2 \notag \\
\hspace{-1em}       + & \frac{c_1 \sigma_{\max}(G_o) \|E^{k+1}\|}{2} \|\phi^{k}-\phi^*\|^2 + \frac{ \sigma_{\max}(G_o)}{2c_1} \|E^{k+1}\| \notag \\
\hspace{-1em}       + & \frac{c_2 \sigma_{\max}(G_o)(\|E^{k}\|+2\|E^{k+1}\|)}{2} \|\phi^{k+1} - \phi^k\|^2 \notag\\
\hspace{-1em}	     + & \frac{\sigma_{\max}(G_o)}{2c_2} (\|E^{k}\|+2\|E^{k+1}\|) \notag \\
\hspace{-1em}       + & \frac{c}{4} \sigma_{\max}^2(G_o) \|E^{k}\|^2 + \frac{c}{4} \sigma_{\max}^2(G_o) \|E^{k+1}\|^2\geq0 .\notag
\end{align}

\textbf{Step 2.} Now we characterize the upper bound of $\|E^{k}\|$. According to the censoring strategy, $\hat{x}_i^{k} - x_i^{k}$, the $i$th block of $E^{k}$, becomes $0$ if $\|\hat{x}_i^{k-1} - x_i^{k}\| \geq \tau^{k}$ or equals $\hat{x}_i^{k-1} - x_i^{k}$ otherwise. In both cases, it holds $\|\hat{x}_i^{k} - x_i^{k}\| \leq \tau^{k}$. Therefore, we know $\|E^{k}\| \leq \sqrt{n} \tau^{k}$. Since $\tau^k$ is non-increasing, it also holds $\|E^{k+1}\| \leq \sqrt{n} \tau^{k+1} \leq \sqrt{n} \tau^{k}$. Thus, \eqref{th1-step1-3} becomes
\begin{align}
&c\|z^k - z^{k+1}\|^2 + \left( \frac{\rho}{2} - \frac{M}{4} \right) \|x^k - x^{k+1}\|^2\label{eq:conv-inner1} \\
&+ \left( \frac{1}{c} - \frac{3c_2 \sigma_{\max}(G_o)\sqrt{n}\tau^k}{2} \right) \|\phi^k-\phi^{k+1}\|^2 \notag \\
\leq& V^k -V^{k+1}+ \frac{c_1 \sigma_{\max}(G_o) \sqrt{n} \tau^{k}}{2} \|\phi^k - \phi^*\|^2  + \notag\\
&\left( \frac1{2c_1} + \frac{3}{2c_2} \right)\sigma_{\max}(G_o) \sqrt{n} \tau^k + \frac{cn\sigma_{\max}^2(G_o)}{2}  ( \tau^{k})^2. \notag
\end{align}

Setting the constants $c_1$ and $c_2$ in \eqref{eq:conv-inner1} as
\begin{align*}
c_1=3c_2=\frac{1}{c\sigma_{\max}(G_o)\sqrt{n} \tau^0},
\end{align*}
we rewrite \eqref{eq:conv-inner1} to
\begin{align} \label{eq:conv-inner1-temp-xxxx}
&c\|z^k - z^{k+1}\|^2 + \big( \frac{\rho}{2} - \frac{M}{4} \big) \|x^k - x^{k+1}\|^2 \\
&+ \big( \frac{1}{c} - \frac{\tau^k}{2c\tau^0} \big) \|\phi^k-\phi^{k+1}\|^2 \notag \\
\leq& V^k -V^{k+1}+ \frac{\tau^{k}}{2c\tau^0} \|\phi^k - \phi^*\|^2  \notag\\
&+ 5 c n\sigma^2_{\max}(G_o) \tau^0 \tau^k  + \frac{cn \sigma_{\max}^2(G_o) ( \tau^{k})^2}{2}. \notag
\end{align}
Since $\tau^k$ is non-decreasing, $\frac{1}{c} - \frac{\tau^k}{2c\tau^0} \geq \frac{1}{2c}$. Meanwhile, by the definition of the energy function, $V^k \geq \frac{1}{c} \|\phi^k - \phi^*\|^2$. By the definitions of $z^k = \frac{1}{2}G_u x^k$ and $L_u = \frac{1}{2}G_u^T G_u$, $\|z^k - z^{k+1}\|^2 = \frac{1}{4} \|G_u(x^k - x^{k+1})\|^2 \geq \frac{1}{2}\lambda_{\min}(L_u) \|x^k - x^{k+1}\|^2$. Applying these three facts to \eqref{eq:conv-inner1-temp-xxxx} yields
\begin{align} \label{eq:conv-inner1-001}
& \frac{1}{2} \big(c \lambda_{\min}(L_u) + \rho - \frac{M}{2} \big) \|x^k - x^{k+1}\|^2 +  \frac{1}{2c} \|\phi^k-\phi^{k+1}\|^2 \notag \\
& \leq \big(1+\frac{\tau^{k}}{2 \tau^0} \big) V^k -V^{k+1}\notag\\
& \hspace{2em} +  5c n\sigma^2_{\max}(G_o) \tau^0 \tau^k + \frac{cn \sigma_{\max}^2(G_o) ( \tau^{k})^2}{2}.
\end{align}

\textbf{Step 3.} Define
\begin{align*}
\theta^k := 5cn\sigma^2_{\max}(G_o) \tau^0 \tau^k + \frac{cn \sigma_{\max}^2(G_o) ( \tau^{k})^2}{2},
\end{align*}
which is a non-increasing non-negative summable summable sequence as $\tau^k$ is. The left-hand side of \eqref{eq:conv-inner1-001} is non-negative because $c \lambda_{\min}(L_u) + \rho \geq \frac{M}{2}$. Thus, \eqref{eq:conv-inner1-001} leads to
\begin{align} \label{eq:conv-inner1-002}
\big(1+\frac{\tau^{k}}{2\tau^0}\big) V^k - V^{k+1} + \theta^k \geq 0.
\end{align}
We use this inequality to show that $V^k$ has a finite upper bound. From (\ref{eq:conv-inner1-002}) we have
	\begin{align}
	&V^{k+1} \leq \big(1+\frac{\tau^{k}}{2\tau^0}\big) V^k + \theta^k \notag\\
	\leq &\big(1+\frac{\tau^{k}}{2\tau^0}\big) \left( \big(1+\frac{\tau^{k-1}}{2\tau^0}\big)
	V^{k-1} + \theta^{k-1}\right) + \theta^k \notag \\
	\leq & \hspace{10em} \ldots\notag\\
	\leq & V^0 \prod_{k'=0}^{k}\big(1+\frac{\tau^{k'}}{2\tau^0}\big)
	+ \sum_{k''=0}^{k-1} \bigg(\theta^{k''} \prod_{k'=k''+1}^{k}
	\big(1+\frac{\tau^{k'}}{2\tau^0}\big)\bigg) + \theta^k \notag \\
	\leq & V^0 \prod_{k'=0}^{k}\big(1+\frac{\tau^{k'}}{2\tau^0}\big)
	+ \sum_{k''=0}^{k} \theta^{k''} \prod_{k'=0}^{k}
	\big(1+\frac{\tau^{k'}}{2\tau^0}\big) \notag \\
	\leq & \big( V^0 + \sum_{k''=0}^{\infty}\theta^{k''} \big)
	\prod_{k'=0}^{\infty}\big(1+\frac{\tau^{k'}}{2\tau^0}\big)\notag \\
	\leq & \big( V^0 + \sum_{k''=0}^{\infty}\theta^{k''} \big)
	\exp\left\{\sum_{k'=0}^{\infty} \frac{\tau^{k'}}{2\tau^0}\right\}
	< \infty,\label{eq:conv-mu}
	\end{align}
where we use the inequality $1+a\leq \exp\{a\}$ that holds for all $a \in \mathcal{R}$, and the fact that $\tau^k$ and $\theta^k$ are both non-negative and summable. Thus, we conclude that $V^k$ has a finite upper bound, denoted as $\bar{V}$.
	
\textbf{Step 4.} Now we begin to prove the convergence. Summing up \eqref{eq:conv-inner1-001} from $k=0$ to $k=\infty$ yields
	\begin{align}
	& \sum_{k=0}^{\infty}
	\bigg[\frac{1}{2}\big( c \lambda_{\min}(L_u) + \rho - \frac{M}{2} \big) \|x^k - x^{k+1}\|^2\notag\\
	&\hspace{2em}+ \frac{1}{2c} \|\phi^k-\phi^{k+1}\|^2 \bigg]
	\nonumber \\
	\leq & V^0 + \sum_{k=0}^{\infty} \frac{\tau^{k}}{2\tau^0} V^k  +  \sum_{k=0}^{\infty}\theta^k \notag\\
	\leq & V^0 + \frac{ \bar{V}}{2\tau^0} \sum_{k=0}^{\infty} \tau^{k}  +  \sum_{k=0}^{\infty}\theta^k < \infty.\label{eq:conv}
	\end{align}
Thus, we conclude that $\lim_{k \rightarrow \infty} (x^k - x^{k+1}) = 0$ and $\lim_{k \rightarrow \infty} (\phi^k-\phi^{k+1}) = 0$. Following these limiting properties, when $k\rightarrow \infty$, the dual update \eqref{update-2} leads to $G_o\hat{x}^k \rightarrow 0$, which implies that
\begin{align}
G_ox^k= G_o\hat{x}^k+G_oE^k&\rightarrow 0. \label{step5-kkt-2}
\end{align}
Also, we have $L_o\hat{x}^k \rightarrow 0$ as $L_o = \frac{1}{2} G_o^T G_o$. Consequently, in the limit \eqref{update-1} becomes
\begin{equation}\label{step5-kkt-1}
\nabla f(x^k)+G_o^T\phi^k \rightarrow 0.
\end{equation}
Meanwhile, by definition
\begin{equation}\label{step5-kkt-3}
\frac{1}{2}G_u  x^k - z^k = 0.
\end{equation}
Comparing \eqref{step5-kkt-1}, \eqref{step5-kkt-2} and \eqref{step5-kkt-3} with the KKT conditions \eqref{kkt-1}, \eqref{kkt-2} and \eqref{kkt-3}, we conclude that	the triple $(x^k, z^k, \phi^k)$ satisfies the KKT conditions when $k$ goes to infinity.
	
Next, we show that $\{(x^k, z^k, \phi^k)\}$ converges when $k\rightarrow \infty$. Since the sequence $V^k$ is bounded, $\|x^k - x^*\|$ and $\|\phi^k - \phi^*\|$ are also bounded. Thus, there exists a subsequence $\{(x^{k_t},\phi^{k_t})\}$ which converges to a cluster point $(x^\infty,\phi^\infty)$ of $\{(x^k, \phi^k)\}$ and $(x^\infty,\phi^\infty)$ is optimal to \eqref{eq:obj-ADM}.
	
Construct another energy function $V_{\infty}^k := \frac{\rho}{2}\|x^k-x^\infty\|^2 + c\|z^k-z^\infty\|^2 +\frac{1}{c}\|\phi^k-\phi^\infty\|^2$, where $z^\infty:=\frac{1}{2}G_u x^\infty$. The analysis for $V^k$ can be applied to $V_{\infty}^k$. In particular, analogous to \eqref{eq:conv-mu}, given any fixed $k_t$, we have
\begin{align}\label{V_infty}
V_{\infty}^k &\leq \bigg( V_{\infty}^{k_t} + \sum_{k''=k_t}^{\infty}\theta^{k''} \bigg)
\exp\bigg\{\sum_{k'=k_t}^{\infty} \frac{ \tau^{k'}} {\tau^{k_t}} \bigg\} \\
&\leq \bigg( V_{\infty}^{k_t} + \sum_{k''=k_t}^{\infty}\theta^{k''} \bigg)
\exp\bigg\{\sum_{k'=k_t}^{\infty} \frac{ \tau^{k'}} {\tau^0} \bigg\},\notag
\end{align}
for any $k\geq k_t$. Observe that $(x^{k_t}, \phi^{k_t}) \rightarrow (x^{\infty}, \phi^{\infty})$ leads to $V_{\infty}^{k_t} \rightarrow 0$. In addition, the sequences $\theta^k$ and $\tau^k$ $\sum_{k''=0}^{\infty}\theta^{k''} < \infty$ and $\sum_{k'=0}^{\infty} \tau^{k'} < \infty$, respectively. Therefore, for any $\epsilon > 0$ there exists an integer $t_0$ such that
\begin{equation}
\hspace{-1em}V_{\infty}^{k_{t_0}} < \frac{\epsilon}{4}, \quad \sum_{k''=k_{t_0}}^{\infty}\theta^{k''} < \frac{\epsilon}{4}, \quad \text{and} \quad \sum_{k'=k_{t_0}}^{\infty} \tau^{k'} < \tau^0 \log2.\notag
\end{equation}
	Then according to (\ref{V_infty}) we have $V_{\infty}^k < \epsilon$ for all $k\geq k_{t_0}$.
	Therefore, $V_{\infty}^k \rightarrow 0$ as $k \rightarrow \infty$.
	From the definition of $V_{\infty}^k$,
	we conclude that $\{(x^k, z^k, \phi^{k})\}$ converges to
	$(x^\infty, z^\infty, \phi^\infty)$, which is optimal to \eqref{eq:obj-ADM}.
\end{proof}

\section{Proof of Theorem \ref{theorem:linear}}
\label{app_linear}

\begin{proof} \textbf{Step 1.} From Assumption \ref{ass:strong}, the local cost functions $f_i$ are strongly convex with constant $m>0$. Thus, we have
\begin{align}
m\|x^{k+1}-x^*\|^2
\leq & \langle \nabla f(x^{k+1})-\nabla f(x^*), x^{k+1}-x^* \rangle \label{thm2step1} \\
= & \langle \nabla f(x^{k})-\nabla f(x^*), x^{k+1}-x^* \rangle \notag\\
&+ \langle \nabla f(x^{k+1})-\nabla f(x^k), x^{k+1}-x^* \rangle. \nonumber
\end{align}

Observe that $\langle \nabla f(x^{k})-\nabla f(x^*), x^{k+1}-x^* \rangle$, the first term at the right-hand side of \eqref{thm2step1}, also appears in \eqref{strong-cor} in the proof of Theorem \ref{theorem:convergence}. We follow the derivation to obtain \eqref{th1-step1-2}, but then look for new upper bounds of $\langle G_oE^{k+1}, \phi^{k}-\phi^* \rangle$ and $\langle  G_oE^{k} , \phi^{k+1} - \phi^k \rangle$, which are different to those in \eqref{th1-step1-2-temp-yyyy-1} and \eqref{th1-step1-2-temp-yyyy-2}. For the term $\langle G_oE^{k+1}, \phi^{k}-\phi^* \rangle$, we observe that
\begin{align} \label{th2-step1-2-temp-yyyy-1}
&\langle G_oE^{k+1}, \phi^{k}-\phi^* \rangle \\
= &\langle G_oE^{k+1}, \phi^{k+1}-\phi^*\rangle + \langle G_oE^{k+1},\phi^{k}-\phi^{k+1} \rangle \notag\\
\leq & \frac{c_1}{2} \|G_oE^{k+1}\|^2 + \frac{1}{2c_1} \|\phi^{k+1}-\phi^*\|^2 \notag\\
&+ \frac{c_1}{2} \|G_oE^{k+1}\|^2 + \frac{1}{2c_1} \|\phi^{k}-\phi^{k+1}\|^2 \nonumber \\
\leq &c_1 \sigma^2_{\max}(G_o) \|E^{k+1}\|^2 \notag\\
&+ \frac{1}{2c_1} \|\phi^{k+1}-\phi^*\|^2 + \frac{1}{2c_1} \|\phi^{k}-\phi^{k+1}\|^2,\notag
\end{align}
where $c_1>0$ is any positive constant. For $\langle  G_oE^{k} , \phi^{k+1} - \phi^k \rangle$, it holds
\begin{align} \label{th2-step1-2-temp-yyyy-2}
&\langle  G_oE^{k}, \phi^{k+1} - \phi^k \rangle\\
\leq& \frac{c_2}{2} \|G_oE^{k}\|^2 + \frac{1}{2c_2} \|\phi^{k+1} - \phi^k\|^2 \notag\\
\leq & \frac{c_2 \sigma^2_{\max}(G_o) \|E^{k}\|^2}{2} + \frac{1}{2c_2} \|\phi^{k+1} - \phi^k\|^2, \notag
\end{align}
where $c_2>0$ is any positive constant.

For the second term at the right-hand side of \eqref{thm2step1}, we have
\begin{align}
&\langle \nabla f(x^{k+1})-\nabla f(x^k), x^{k+1}-x^* \rangle\label{th2-step1-2}\\
\leq& \frac{c_3}{2}\|\nabla f(x^{k+1})-\nabla f(x^k)\|^2 + \frac{1}{2c_3}\|x^{k+1}-x^*\|^2 \notag \\
\leq& \frac{c_3M^2}{2}\|x^{k+1}-x^k\|^2 + \frac{1}{2c_3}\|x^{k+1}-x^*\|^2,\notag
\end{align}
where $c_3>0$ is any positive constant. The last inequality uses the fact that the gradients of the local cost functions $\nabla f_i$ are Lipschitz continuous with constant $M>0$ according to Assumption \ref{ass:Lip}.

Using \eqref{th2-step1-2-temp-yyyy-1}, \eqref{th2-step1-2-temp-yyyy-2} and \eqref{th1-step1-2-temp-yyyy-3} to rewrite \eqref{th1-step1-2} followed by substituting the result and \eqref{th2-step1-2} into \eqref{thm2step1}, we obtain
\begin{align}
& V^{k+1} \leq V^k - c\|z^k - z^{k+1}\|^2 \label{th2-step1-3-temp}\\
&- \big( \frac{\rho}{2} - \frac{c_3M^2}{2} \big)\|x^k - x^{k+1}\|^2 \notag\\
&- \big( \frac{1}{c} - \frac{1}{2c_1} - \frac{1}{2c_2} \big) \|\phi^k-\phi^{k+1}\|^2 \notag \\
&- \big( m-\frac{1}{2c_3} \big)\|x^{k+1} - x^*\|^2
+ \frac{1}{2c_1}\|\phi^{k+1} - \phi^*\|^2 \nonumber \\
&+ \big(\frac{c_2}{2} + \frac{c}{4} \big)\sigma^2_{\max}(G_o) \|E^{k}\|^2 + \big(c_1 + \frac{c}{4} \big)\sigma^2_{\max}(G_o) \|E^{k+1}\|^2. \notag
\end{align}
By the same reasoning in the proof of Theorem \ref{theorem:convergence}, $\|E^{k}\| \leq \sqrt{n} \tau^k$ and $\|E^{k+1}\| \leq \sqrt{n} \tau^k$. Thus, \eqref{th2-step1-3-temp} becomes
\begin{align} \label{th2-step1-3}
& V^{k+1} \leq V^k - c\|z^k - z^{k+1}\|^2 \\
&- \big( \frac{\rho}{2} - \frac{c_3M^2}{2} \big)\|x^k - x^{k+1}\|^2 \notag\\
&- \big( \frac{1}{c} - \frac{1}{2c_1} - \frac{1}{2c_2} \big) \|\phi^k-\phi^{k+1}\|^2 \notag \\
&- \big( m-\frac{1}{2c_3} \big)\|x^{k+1} - x^*\|^2 + \frac{1}{2c_1}\|\phi^{k+1} - \phi^*\|^2 + s n (\tau^k)^2. \nonumber
\end{align}
where $$s := \left(c_1 + \frac{c_2}{2} + \frac{c}{2} \right) \sigma^2_{\max}(G_o) > 0.$$
	
\textbf{Step 2.} Now we are going to find constants $\delta>0$ and $ \gamma \geq 0$ such that
\begin{equation}\label{thm2step2aim}
(1+\delta)V^{k+1} \leq V^k+ \gamma n (\tau^k)^2.
\end{equation}

Given any $\delta$, using the definition of the energy function $V^{k+1}$ to rewrite \eqref{th2-step1-3} as
\begin{align} \label{th2-step2-1}
& (1+\delta)V^{k+1} \leq V^k - c\|z^k - z^{k+1}\|^2 \\
&- \big( \frac{\rho}{2} - \frac{c_3M^2}{2} \big)\|x^k - x^{k+1}\|^2\notag \\
&- \big( \frac{1}{c} - \frac{1}{2c_1} - \frac{1}{2c_2} \big) \|\phi^k-\phi^{k+1}\|^2 \notag\\
&- \big( m-\frac{1}{2c_3} - \frac{\rho \delta}{2} \big)\|x^{k+1} - x^*\|^2 \notag\\
&+ \big(\frac{1}{2c_1} + \frac{\delta}{c}\big) \|\phi^{k+1} - \phi^*\|^2 + c \delta \|z^{k+1}-z^*\|^2 + s n (\tau^k)^2. \notag
\end{align}
We shall replace the terms $\|z^{k+1}-z^*\|^2$ and $\|\phi^{k+1}-\phi^*\|^2$ in \eqref{th2-step2-1} with terms $\|z^k - z^{k+1}\|^2$, $ \|z^{k+1} - z^*\|^2$, $\|x^k - x^{k+1}\|^2$, $\|x^{k+1} - x^*\|^2$ and $(\tau^k)^2$.

For $\|z^{k+1}-z^*\|^2$, because $z^{k+1}-z^* = \frac{G_u}{2} (x^{k+1}-x^*)$, we have
\begin{equation}
\|z^{k+1}-z^*\|^2 \leq \frac{\sigma^2_{\max}(G_u)}{4} \|x^{k+1} - x^*\|^2.\label{th2-step2-2}
\end{equation}

To handle $\|\phi^{k+1}-\phi^*\|^2$, use the fact that $L_u = \frac{1}{2}G_u^T G_u$ and reorganize \eqref{lemma2-1} to obtain
\begin{align} \label{lemma2-1-thm2}
&G_o^T(\phi^{k+1}-\phi^*) \\
=& - \left( \nabla f(x^{k}) - \nabla f(x^{*}) \right) + cG_u^T(z^k - z^{k+1}) \notag\\
&+ \rho(x^k - x^{k+1}) + cL_o (E^k-E^{k+1}).\notag
\end{align}
Since both $\phi^{k+1}$ and $\phi^*$ are in the column space of $G_o$, the left-hand side of \eqref{lemma2-1-thm2} is lower-bounded by
\begin{align}
\tilde{\sigma}^2_{\min}(G_o)\|\phi^{k+1}-\phi^*\|^2 \leq \|G_o^T(\phi^{k+1}-\phi^*)\|^2. \label{th2-step2-2-temp-xxxx}
\end{align}
The right-hand side of \eqref{lemma2-1-thm2} is upper-bounded by
\begin{align}\label{th2-step2-2-temp-yyyy}
      &\|- \left( \nabla f(x^{k}) - \nabla f(x^{*}) \right) + cG_u^T(z^k - z^{k+1}) \\
     &\hspace{.5em} + \rho(x^k - x^{k+1}) + cL_o (E^k-E^{k+1})\|^2  \notag\\
\leq  &4\|\nabla f(x^{k})-\nabla f(x^*)\|^2 + 4\|cG_u^T(z^k - z^{k+1})\|^2 \notag\\
&+ 4\|\rho(x^k - x^{k+1})\|^2 + 4\|cL_o (E^k-E^{k+1})\|^2 \notag \\
	\leq  &8\|\nabla f(x^{k+1})-\nabla f(x^*)\|^2 + 8\|\nabla f(x^{k})-\nabla f(x^{k+1})\|^2\notag\\	
	&+ 4\|cG_u^T(z^k - z^{k+1})\|^2  + 4\| \rho(x^k - x^{k+1})\|^2 \notag\\
	&+ 8  \|cL_o E^k\|^2 + 8 \|cL_oE^{k+1}\|^2 \notag \\
	\leq  &8M^2 \|x^{k+1}-x^*\|^2 + (8M^2 + 4\rho^2) \|x^{k}-x^{k+1}\|^2 \notag\\
	&+ 4c^2\sigma^2_{\max}(G_u) \|z^k - z^{k+1}\|^2 + 4c^2\sigma^4_{\max}(G_o) n (\tau^k)^2. \notag
\end{align}
The last inequality uses the fact that $\nabla f$ is Lipschitz continuous with constant $M > 0$ such that $$ \|\nabla f(x^{k+1})-\nabla f(x^*)\|^2 \leq M^2 \|x^{k+1}-x^*\|^2,$$  $$\|\nabla f(x^{k})-\nabla f(x^{k+1})\|^2 \leq M^2 \|x^{k}-x^{k+1}\|^2,$$ and the definition of $L_o = \frac{1}{2}G_o^T G_o$ such that $$\|cL_o E^k\|^2 \leq \frac{c^2}{4} \sigma^2_{\max}(G_o) \|E^k\|^2 \leq \frac{c^2}{4} \sigma^4_{\max}(G_o) n (\tau^k)^2,$$ $$\|cL_o E^{k+1}\|^2 \leq \frac{c^2}{4} \sigma^4_{\max}(G_o) n (\tau^k)^2.$$ Combining \eqref{lemma2-1-thm2}, \eqref{th2-step2-2-temp-xxxx} and \eqref{th2-step2-2-temp-yyyy}, we obtain
\begin{align}\label{th2-step2-3}
& \|\phi^{k+1}-\phi^*\|^2 \leq \frac{1}{\tilde{\sigma}^2_{\min}(G_o)}
\left( 8M^2 \|x^{k+1}-x^*\|^2 \right.  \\
&+ (8M^2 + 4\rho^2) \|x^{k}-x^{k+1}\|^2 + 4c^2\sigma^2_{\max}(G_u) \|z^k - z^{k+1}\|^2 \nonumber \\
&+ \left.4c^2\sigma^4_{\max}(G_o) n (\tau^k)^2 \right).\notag
\end{align}

Thus, we can use \eqref{th2-step2-2} and \eqref{th2-step2-3} to rewrite \eqref{th2-step2-1} as
\begin{align}\label{thm2step2}
&(1+\delta)V^{k+1}\\ \leq& V^k - c\bigg(1-\big(\frac{1}{2c_1} + \frac{\delta}{c} \big) \frac{4c\sigma^2_{\max}(G_u)}{\tilde{\sigma}^2_{\min}(G_o)}\bigg) \|z^k - z^{k+1}\|^2 \notag \\
& \hspace{-1.5em} - \bigg( \frac{\rho}{2} - \frac{c_3M^2}{2} - \big(\frac{1}{2c_1} + \frac{\delta}{c} \big)\frac{(8M^2 + 4\rho^2)}{\tilde{\sigma}^2_{\min}(G_o)} \bigg) \|x^k - x^{k+1}\|^2 \notag\\
& \hspace{-1.5em} - \bigg( \frac{1}{c} - \frac{1}{2c_1} - \frac{1}{2c_2} \bigg) \|\phi^k-\phi^{k+1}\|^2 \notag \\
& \hspace{-1.5em} - \bigg( m-\frac{1}{2c_3}-\frac{\rho\delta}{2} - \frac{c\delta\sigma^2_{\max}(G_u)}{4} - \big(\frac{1}{2c_1} + \frac{\delta}{c} \big)\frac{8M^2}{\tilde{\sigma}^2_{\min}(G_o)} \bigg)\cdot \notag\\
&\hspace{1.5em} \|{x}^{k+1} - x^*\|^2 \notag \\
& \hspace{-1.5em} + \bigg( s + \big(\frac{1}{2c_1} + \frac{\delta}{c} \big)\frac{4c^2\sigma^4_{\max}(G_o)}{\tilde{\sigma}^2_{\min}(G_o)} \bigg) n (\tau^k)^2. \nonumber
\end{align}

For convenience, set the constants as
\begin{align*}
c_1=&\frac{c}{2} + \frac1{m(2m\rho-M^2)\tilde{\sigma}^2_{\min}(G_o)}
\bigg(4M^2(2m\rho+M^2) \\
&+ 16m^2(2M^2+\rho^2) + 2c\sigma^2_{\max}(G_u)m(2m\rho-M^2)\bigg), \quad\\
c_2=&\frac{c_1c}{2c_1-c}, \quad\\
c_3=&\frac{1/2m+\rho/M^2}{2}=\frac{2m\rho+M^2}{4mM^2}
\in \left(\frac{1}{2m},\frac{\rho}{M^2}\right),
\end{align*}
where the range of $c_3$ is from the hypothesis that $\rho > \frac{M^2}{2m}$. Then (\ref{thm2step2aim}) is achieved with constants
\begin{align*}
\delta &\leq \min\bigg\{ \frac{\tilde{\sigma}^2_{\min}(G_o)}{4\sigma^2_{\max}(G_u)}-\frac{c}{2c_1},
\frac{c\tilde{\sigma}^2_{\min}(G_o)(2m\rho-M^2)}{32m(\rho^2+2M^2)}-\frac{c}{2c_1},\\
&\hspace{4em}\frac{\frac{m(2m\rho-M^2)}{2m\rho+M^2}-\frac{4M^2}{c_1\tilde{\sigma}^2_{\min}(G_o)}}
{\big( \frac{\rho}{2}+\frac{c\sigma^2_{\max}(G_u)}{4}
	+\frac{8M^2}{c\tilde{\sigma}^2_{\min}(G_o)} \big)} \bigg\}, \\
\gamma &=s + \left(\frac{\delta}{c} + \frac{1}{2c_1}\right)
\frac{4c^2\sigma^4_{\max}(G_o)}{\tilde{\sigma}^2_{\min}(G_o)}.
\end{align*}
Note that $\delta>0$ and $\gamma>0$.
	
\textbf{Step 3.} Now we prove the linear convergence of $V^k$ to $0$, which implies the linear convergence of $x^k$ to $x^*$. Using the censoring threshold rule $\tau^k = \alpha\cdot(\beta)^k$, we further rewrite (\ref{thm2step2aim}) as
\begin{equation}
(1+\delta)V^{k+1} \leq V^k + \gamma n\alpha^2\cdot(\beta^2)^k.\notag
\end{equation}
Analogous to the technique used in handling \eqref{eq:conv-mu}, it holds
\begin{align}
V^{k+1} & \leq (1+\delta)^{-1} \left(V^k + \gamma n\alpha^2\cdot(\beta^2)^k \right) \notag \\
& \hspace{-3em} \leq (1+\delta)^{-1} \left[ (1+\delta)^{-1} \left(V^{k-1} + \gamma n\alpha^2\cdot(\beta^2)^{k-1}\right)+  \gamma n\alpha^2\cdot(\beta^2)^k \right] \notag \\
& \hspace{-3em} \leq \hspace{10em} \ldots\notag\\
& \hspace{-3em} \leq (1+\delta)^{-(k+1)} V^0 	+ Cn\alpha^2\sum_{k'=0}^{k} \left( (1+\delta)^{-(k+1-k')} (\beta^2)^{k'} \right) \notag \\
& \hspace{-3em} = (1+\delta)^{-(k+1)} \left[V^0 + \gamma n\alpha^2 \sum_{k'=0}^{k} \left((1+\delta)\beta^2\right)^{k'} \right] \notag \\
& \hspace{-3em} \leq (1+\delta)^{-(k+1)} \left(V^0 + \frac{\gamma n\alpha^2}{1-(1+\delta)\beta^2} \right),\notag
\end{align}
where the last inequality holds when $(1+\delta)\beta^2<1$.

In summary, for any positive $\delta>0$ that satisfies
\begin{align}\label{eq:delta-bound}
\hspace{-1.5em}\delta \leq \min\bigg\{ &\frac{\tilde{\sigma}^2_{\min}(G_o)}{4\sigma^2_{\max}(G_u)}-\frac{c}{2c_1},
 \frac{c\tilde{\sigma}^2_{\min}(G_o)(2m\rho-M^2)}{32m(\rho^2+2M^2)}-\frac{c}{2c_1},\notag\\
&\frac{\frac{m(2m\rho-M^2)}{2m\rho+M^2}-\frac{4M^2}{c_1\tilde{\sigma}^2_{\min}(G_o)}}
{\big( \frac{\rho}{2}+\frac{c\sigma^2_{\max}(G_u)}{4}
+\frac{8M^2}{c\tilde{\sigma}^2_{\min}(G_o)} \big)},
\frac{1}{\beta^2}-1 \bigg\},
\end{align}
the energy function $V^k$ converges to $0$ at a linear rate of $\mathcal{O}((1+\delta)^{-k})$. Moreover, by the definition of $V^k$, it holds that $V^k \geq \frac{\rho}{2}\|x^k-x^*\|^2$. Thus, the primal variable $x^k$ converges to the unique optimal solution $x^*$ at $\mathcal{O}((1+\delta)^{-\frac{k}{2}})$.
\end{proof}

\begin{remark}\label{rmk:cor1}
If we set $c=\frac{8M}{\sigma_{\max}(G_u)\tilde\sigma_{\min}(G_o)}$ and $\rho=M \kappa_f$ and further set $\frac1{2c_1}=\frac\delta c$, $\frac1{2c_2}=\frac1c - \frac1{2c_1}$ and $c_3=\frac2{3m}$ in \eqref{thm2step2}, then following the proof of Theorem \ref{theorem:linear}, we can derive another upper bound of $\delta$ as shown in \eqref{eq:cor} of Corollary \ref{cor:delta}.
\end{remark}

\section{Proof of Theorem \ref{theorem:sublin}}
\label{app_sublin}

\begin{proof}
\textbf{Step 1.} As in Step 1 of the proof of Theorem \ref{theorem:linear}, we obtain the inequality \eqref{th2-step1-3}.	
	
\textbf{Step 2.} Now our aim is different to that in Step 2 of the proof of Theorem \ref{theorem:linear}, as we are going to find constants $q>0$ and $\eta^k\geq 0$ as well as a time index $k_0$ such that
\begin{equation}\label{thm3step2aim}
(k+1)^q\left(V^{k+1}+\eta^{k+1}\right) \leq (k)^q\left(V^k+ \eta^k\right),
\end{equation}
for all $k\geq k_0$.
	
From \eqref{thm3step2aim}, in which $k_0$, $q$ and $\eta^k$ will be determined later, we have
\begin{align*}
&(k+1)^q\left(V^{k+1}+\eta^{k+1}\right) \\
=& (k)^qV^{k+1}+ \left[(k+1)^q-(k)^q\right]
\big( \frac{\rho}{2}\|x^{k+1}-x^*\|^2 \\
&+ c\|z^{k+1}-z^*\|^2 +\frac{1}{c}\|\phi^{k+1}-\phi^*\|^2 \big) + (k+1)^q \eta^{k+1}\\
\overset{\eqref{th2-step1-3}}{\leq}&
(k)^qV^k - c(k)^q\|z^k - z^{k+1}\|^2 \\
&- (k)^q\big( \frac{\rho}{2} - \frac{c_3M^2}{2} \big)\|x^k - x^{k+1}\|^2 \\
&- (k)^q\big( \frac{1}{c} - \frac{1}{2c_1} - \frac{1}{2c_2} \big) \|\phi^k-\phi^{k+1}\|^2 \notag \\
&-\bigg\{(k)^q\big(m-\frac{1}{2c_3}\big) - \big[(k+1)^q-(k)^q\big]\frac{\rho}{2}\bigg\} \cdot\\
&\hspace{2em}\|x^{k+1} - x^*\|^2\\
&+ c\left[(k+1)^q-(k)^q\right] \|z^{k+1} - z^*\|^2 \\
&+ \bigg[ \frac{(k)^q}{2c_1} + \frac{(k+1)^q-(k)^q}{c} \bigg] \|\phi^{k+1} - \phi^*\|^2 \\
&+ (k)^q s n(\tau^k)^2 + (k+1)^q \eta^{k+1} \\
\overset{\eqref{th2-step2-2},\eqref{th2-step2-3}}{\leq}& (k)^qV^k
- (k)^q \bigg[c - \frac{1}{2c_1} - \frac{(k+1)^q-(k)^q}{(k)^q}\cdot\\
&\hspace{6.5em}\frac{4c\sigma^2_{\max}(G_u)}{\tilde{\sigma}^2_{\min}(G_o)}\bigg] \|z^k - z^{k+1}\|^2 \\
&- (k)^q \bigg\{ \frac{\rho}{2} - \frac{c_3M^2}{2} - \big(\frac{c}{2c_1} + \frac{(k+1)^q-(k)^q}{(k)^q} \big)\cdot\\
&\hspace{3.5em}\frac{8M^2 + 4\rho^2}{c\tilde{\sigma}^2_{\min}(G_o)} \bigg\}\|x^k - x^{k+1}\|^2 \\
&- (k)^q\bigg[ m-\frac{1}{2c_3} -\frac{4M^2}{c_1\tilde{\sigma}^2_{\min}(G_o)}
- \frac{(k+1)^q-(k)^q}{(k)^q} \cdot\\
&\hspace{3.5em}\big(\frac{\rho}{2} + \frac{c\sigma^2_{\max}(G_u)}{4}
+ \frac{8M^2}{c\tilde{\sigma}^2_{\min}(G_o)} \big) \bigg]
\|x^{k+1} - x^*\|^2 \\
&- (k)^q\big( \frac{1}{c} - \frac{1}{2c_1} - \frac{1}{2c_2} \big) \|\phi^k-\phi^{k+1}\|^2 \\
&+ (k)^q t^k n (\tau^k)^2 + (k+1)^q \eta^{k+1},
\end{align*}
where $$t^k:=s + \frac{2c^2\sigma^4_{\max}(G_o)}{c_1\tilde{\sigma}^2_{\min}(G_o)} + \frac{(k+1)^q-(k)^q}{(k)^q}\frac{4c\sigma^4_{\max}(G_o)}{\tilde{\sigma}^2_{\min}(G_o)}>0.$$
	
Set the constants $c_1$, $c_2$ and $c_3$ the same values as those in the proof of Theorem \ref{theorem:linear}. Notice that $\frac{(k+1)^q-(k)^q}{(k)^q}=\left(1+\frac{1}{k}\right)^q-1\rightarrow0$ as $k$ goes to infinity. Then, there exists a time index $k_0$ such that for any $k \geq k_0$, it holds
\begin{align}
&\frac{(k+1)^q-(k)^q}{(k)^q} \leq \min\bigg\{ \frac{\tilde{\sigma}^2_{\min}(G_o)}{4\sigma^2_{\max}(G_u)} \big(c-\frac{1}{2c_1}\big), \notag\\
&\hspace{3em}\frac{c\tilde{\sigma}^2_{\min}(G_o)}{4( \rho^2+ 2M^2)} \big( \frac{\rho}{2} - \frac{c_3M^2}{2}
-  \frac{2(\rho^2+2M^2)}{c_1\tilde{\sigma}^2_{\min}(G_o)} \big),\notag\\
&\hspace{3em} \frac{m-\frac{1}{2c_3} -\frac{4M^2}{c_1\tilde{\sigma}^2_{\min}(G_o)}}
{\frac{\rho}{2} + \frac{c\sigma^2_{\max}(G_u)}{4} + \frac{8M^2}{c\tilde{\sigma}^2_{\min}(G_o)}},
\frac{\tilde{\sigma}^2_{\min}(G_o)}{4c\sigma^4_{\max}(G_o)}  \bigg\}, \label{thm3-k_bound}
\end{align}
where the right-hand side is larger than $0$. In this situation, $t^k\leq t := s + \frac{2c^2\sigma^4_{\max}(G_o)}{c_1\tilde{\sigma}^2_{\min}(G_o)} +1$.	Further, using the censoring threshold $\tau^{k} = \frac{\alpha}{(k)^r}$, we have
\begin{equation}\label{thm3step2aim2}
\hspace{-2em}(k+1)^q\big(V^{k+1}+\eta^{k+1}\big) \leq (k)^qV^k + \frac{tn\alpha^2}{(k)^{2r-q}} + (k+1)^q \eta^{k+1}.
\end{equation}

Now we determine the values of $q$ and $\eta^k$. Since $\sum_{k'=k}^\infty\frac{1}{(k')^{2r-q}}< \infty$ for any time index $k $ when $2r-q>1$, setting $(k)^q \eta^k:=\sum_{k'=k}^\infty\frac{tn\alpha^2}{(k')^{2r-q}}$ in \eqref{thm3step2aim2} leads to an equivalent form
\begin{equation}\notag
(k+1)^q\left(V^{k+1}+ \eta^{k+1}\right) \leq (k)^q\left(V^k+ \eta^k\right),
\end{equation}
which is exactly what we want in \eqref{thm3step2aim}. Therefore, for any $k \geq k_0$, it holds
$$V^{k}\leq V^{k}+\eta^{k} \leq \frac{(k_0)^q\left(V^{k_0}+\eta^{k_0}\right)}{(k)^q}.$$
That is, the energy function $V^k$ converges to $0$ at a sublinear rate of $\mathcal{O}((k)^{-q})$. Moreover, by the definition of $V^k$, it holds that $V^k \geq \frac{\rho}{2}\|x^k-x^*\|^2$. Thus, the primal variable $x^k$ converges to the unique optimal solution $x^*$ at a sublinear rate of $\mathcal{O}((k)^{-\frac{q}{2}})$.
\end{proof}

\end{document}